\documentclass[preprint,1p,times]{elsarticle}

\usepackage{amsbsy}
\usepackage{bm}
\usepackage{caption}
\usepackage{subcaption}
\usepackage[mathscr]{eucal}
\usepackage[all]{xy}
\usepackage{graphicx,color}
\usepackage[centertags]{amsmath}
\usepackage{amsfonts,amssymb,amsthm,mathrsfs, mathtools}
\usepackage{newlfont, verbatim}
\usepackage{hyperref}
\usepackage{geometry}
\usepackage{float}
\usepackage{setspace}
\usepackage{portland}
\usepackage{wrapfig}
\usepackage[linesnumbered, commentsnumbered]{algorithm2e}

\newtheorem{assumption}{Assumption}
\newtheorem{remark}{Remark}
\newtheorem{lemma}{Lemma}
\newtheorem{theorem}{Theorem}
\newtheorem{example}{Example}

\newcommand{\R}{\mathbb{R}}
\newcommand{\N}{\mathbb{N}}
\newcommand{\E}{\mathbb{E}}

\newcommand{\Q}{\QQQ}
\newcommand{\C}{\mathcal{C}}
\newcommand{\mix}{\mathrm{mix}}

\newcommand{\vecy}{{\vec\yyy}}
\newcommand{\yyy}{{y}}
\def\bx{{\mathbf x}}
\def\NNN{{N}}
\def\nnn{{n}}
\def\MMM{{M}}
\def\mmm{{m}}

\def\QQQ{{Q}}

\journal{Computers \& Mathematics with Applications}

\begin{document}

\begin{frontmatter}



\title{Multilevel Sparse Grid Methods for Elliptic Partial Differential Equations with Random Coefficients}

\author[rvt]{H.-W.~van Wyk\corref{cor1}}
\ead{hvanwyk@fsu.edu}




\address[rvt]{Florida State University, Department of Scientific Computing, 400 Dirac Science Library,Tallahassee, FL 32306-4120}

\begin{abstract}
\noindent Stochastic sampling methods are arguably the most direct and least intrusive means of incorporating parametric uncertainty into numerical simulations of partial differential equations with random inputs. However, to achieve an overall error that is within a desired tolerance, a large number of sample simulations may be required (to control the sampling error), each of which may need to be run at high levels of spatial fidelity (to control the spatial error). Multilevel sampling methods aim to achieve the same accuracy as traditional sampling methods, but at a reduced computational cost, through the use of a hierarchy of spatial discretization models. Multilevel algorithms coordinate the number of samples needed at each discretization level by minimizing the computational cost, subject to a given error tolerance. They can be applied to a variety of sampling schemes, exploit nesting when available, can be implemented in parallel and can be used to inform adaptive spatial refinement strategies. We extend the multilevel sampling algorithm to sparse grid stochastic collocation methods, discuss its numerical implementation and demonstrate its efficiency both theoretically and by means of numerical examples.
\end{abstract}


\end{frontmatter}

\section{Introduction}

Computing has become an invaluable tool in modern science and engineering because, increasingly, computer simulations are used to supplement or replace experiments and prototype engineering systems, and to predict the behavior of complex physical processes. Often, however, the precise environmental conditions (or model parameters) surrounding the process that is being simulated are known only with a limited degree of certainty. For systems governed by partial differential equations (PDEs) with random inputs, statistical sampling methods present arguably the most direct and least intrusive means of incorporating parametric uncertainty into numerical simulations. Descriptive statistics related to the random simulation output are obtained by generating representative samples of the input parameters and then running the numerical simulation for each sample point, yielding sample of outputs that can then be aggregated statistically.

\vspace{1em}

To be more specific, let $(\Omega,\mathcal F, \mathbb P)$ denote the complete probability space underlying the system's uncertain input parameters. For any sample point $\omega\in \Omega$ corresponding to a given system configuration, let $u(\bx,\omega)$ denote the resulting simulation output and let $G_1\big(u(\bx,\omega)\big)$ denote a physical output of interest (e.g., a function value, a spatial average, the total energy, or the flux across a boundary) that is determined from $u(\bx,\omega)$.\footnote{Of course, the simulation output could also depend on time, but for the sake of simplicitly, we suppress mention of such possible dependences.}   A large class of statistical quantities of interest $\QQQ$ associated with an output of interest $G_1\big(u(\bx,\omega)\big)$ take the form of a stochastic integral or expectation, i.e.,
\begin{equation}\label{eqn: intro statistical qoi 1}
  \QQQ := \E\big[G_2\big(G_1(u)\big)\big] = \int_\Omega
   G_2\Big(G_1\big(u(\bx,\omega)\big)\Big)\,d\mathbb{P}(\omega) 
\end{equation}
for an appropriate choice of $G_2$; for example, if $G_2(v)= v^k$, then $\QQQ$ is the $k^{th}$ raw statistical moment of $G_1(u)$ or, if $G_2(v) = \chi_{\{G_1(u)\geq a)\}}$, where $\chi$ is the characteristic function, then $\QQQ$ equals the exceedance probability $\mathbb{P}\big[G_1\big(u(\omega)\big) \geq a\big]$. Because this paper addresses the numerical approximation of the integral \eqref{eqn: intro statistical qoi 1}, it is not essential for us to know the details about how the integrand is constructed from the output of interest $G_1$ and the desired statistical information embodied in $G_2$. Thus, we can refer directly to the integrand by letting $G = G_2\circ G_1$ so that we rewrite \eqref{eqn: intro statistical qoi 1} as
\begin{equation}\label{eqn:intro statistical qoi 2}
\QQQ = \E[G(u)] = \int_\Omega G\big(u(\bx,\omega)\big) \, d\mathbb{P}(\omega).
\end{equation}

\vspace{1em}

In general, input functions that are modeled as spatially varying random fields are first approximated by functions of a finite-dimensional random parameter vector $\vecy(\omega):= \big(\yyy_1(\omega),\ldots, \yyy_N(\omega)\big)$ with range in some hyper-rectangle $\Gamma = \prod_{n=1}^N \Gamma_n\subset \R^N$ and known joint probability density function $\rho:\Gamma \rightarrow [0,\infty)$. Such ``finite noise'' approximations may be achieved through an expansion in terms of piecewise constant functions based on a subdivision of the spatial domain, or through truncated spectral expansions related to the field's correlation function, e.g., via Karhunen-Lo\`eve expansions; see \cite{Karhunen1947, Loeve1978, Schwab2006}. Under this approximation, the statistical quantity of interest $\QQQ$ given by \eqref{eqn:intro statistical qoi 2} takes the form of a high-dimensional integral, i.e.,
\begin{equation}\label{eqn:intro_stat_qoi_3}
\QQQ =  \E[G(u)] = \int_\Gamma G\big(u(\bx,\vecy)\big) \rho(\vecy)\,d\vecy ,
\end{equation}
where $\vecy$ denotes the vector of random parameters.

\vspace{1em}

In practice, for any $\vecy\in\Gamma$, only spatial approximations $u_h(\bx,\vecy)$ (determined via, e.g., finite element, finite difference, finite volume, or spectral methods) of the solution $u(\bx,\vecy)$ are available. Here $h$ is a spatial discretization parameter that is often related to the spatial grid size. As a result, instead of \eqref{eqn:intro_stat_qoi_3}, one can only determine the approximation
\begin{equation}\label{eqn:intro_stat_qoi_3a}
\QQQ\approx\QQQ_h :=   \E[G(u_h)]=  \int_\Gamma G\big(u_h(\bx,\vecy)\big) \rho(\vecy)\,d\vecy 
\end{equation}
of the quantity of interest $\QQQ$.

\vspace{1em}

A statistical sampling method is simply a numerical quadrature scheme that estimates the statistical quantity of interest given by \eqref{eqn:intro_stat_qoi_3} or \eqref{eqn:intro_stat_qoi_3a} by a quadrature rule, e.g., in the latter case, a weighted sum of the form
\begin{equation}\label{eqn:qoi_full_approx}
\QQQ\approx \QQQ_{\MMM,h} :=  \sum_{\mmm=1}^{\MMM} \mu_\mmm G\big( u_h(\bx,\vecy^\mmm)\big),
\end{equation}
where $\{\yyy^{(\mmm)}\}_{\mmm=1}^\MMM$ denotes a collection of samples of $\vecy\in\Gamma$ and $\{\mu_\mmm\}_{\nnn=1}^\NNN$ a given set of weights. Note the evaluation of $ \QQQ_{\MMM,h}$ requires $\MMM$ solutions $\{u_h(\bx,\vecy^\mmm)\}_{\mmm=1}^\MMM$ of the discretized PDE, one for each of the $\MMM$ samples $\vecy^\mmm$ of the parameter vector $\vecy$. Depending on the statistical complexity of the underlying parametric uncertainty and on the sampling scheme used, an accurate approximation $\QQQ_{\MMM,h}$ of $\QQQ$ may require a large number of simulation runs, i.e., $\MMM$ may be large; clearly, this can be computationally intensive, especially when individual simulations are run at a high level of spatial fidelity, i.e., for small $h$. Increasing $\NNN$, i.e., increasing the dimension of the parameter space, especially results in explosive growth in computational complexity, a phenomena commonly referred as the {\em curse of dimensionality.}

\vspace{1em}

Monte Carlo (MC) sampling provides a straightforward means of approximating the integral in \eqref{eqn:intro_stat_qoi_3} by generating $\MMM$ random samples $\vecy^\mmm\in\Gamma$, $\mmm=1,\ldots,\MMM$, based on the PDF $\rho(\vecy)$ and then simply averaging the resulting $G\big(u_h(\bx,\vecy^\mmm)\big)$. Thus, $\mu_\mmm=1/\MMM$ for all $\mmm$ and \eqref{eqn:qoi_full_approx} becomes
\begin{equation}\label{mcm}
\QQQ\approx \QQQ_{\MMM,h}^{\mathrm{MC}} = 
 \frac{1}{\MMM} \sum_{\mmm=1}^\MMM G\big(u_h(\bx,\vecy^\mmm)\big).
\end{equation}

Although the MC method is largely immune from the curse or dimensionality, its suffers from very slow convergence with respect to increasing $\MMM$. In fact, the rate at which the root mean squared error converges is $O(M^{-1/2})$. This has motivated the development of multilevel Monte Carlo (MLMC) methods. These methods aim to achieve the same accuracy as traditional MC methods but at a reduced computational cost by making use of a hierarchy of spatial simulation models having increasing fidelity, e.g., based on decreasing values of $h$. The MC method as described by \eqref{mcm} uses a single spatial model, i.e., a single value of $h$. MLMC methods were first introduced in \cite{Heinrich2001} for the evaluation of parametric integrals, especially those arising from the approximation of integral equations. In \cite{Giles2008, Giles2008a, Giles2009}, the algorithm is further developed, extending its application to numerical simulations of stochastic differential equations related to computational finance. In \cite{Barth2010}, a version of the method was adapted to finite element approximations of elliptic partial differential equations with stochastic inputs. There, the sample sizes were chosen to equilibrate the sampling and spatial discretization errors at each refinement level, resulting in approximations of $\QQQ$ that, in certain cases, are of log-linear complexity. This approach was generalized to include a variety of other stochastic sampling schemes in \cite{Harbrecht2013}, where its behavior was explained through analogies with sparse-grid methods \cite{Bungartz2004}.

\vspace{1em}

In \cite{Cliffe2011}, an altogether more conceptual view was taken by examining the MLMC method as a numerical optimization problem. The number of parameter samples needed at each spatial discretization level are coordinated so as to minimize the total computational cost, subject to a given error tolerance. Simulations based on smaller values of $h$ are sampled sparingly, whereas those based on coarser grids form the bulk of the sampling, where possible. This framework lends a certain degree of flexibility to the MLMC method by allowing for the incorporation of different spatial error estimates and statistical quantities of interest \cite{Teckentrup2012,Charrier2011} as well as other factors that may influence the convergence rate such as the truncation level of the Karhunen-Lo\`eve expansion, parallel implementations, and quadrature nesting. 

\vspace{1em}

An alternative to sampling methods such as MC or quasi-MC methods for approximating the quantity of interest $Q$ are provided by interpolatory methods which are often referred to as {\em stochastic collocation} (SC) methods. In this setting, the parameter dependence of the spatial approximation $u_h(\bx,\vecy)$ is itself approximated in a finite dimensional space $V_\MMM(\Gamma)$ which is spanned by a set of interpolatory basis functions $\{\psi_\mmm\}_{\mmm=1}^{\MMM}$ that correspond to a predetermined, i.e., deterministic, set of sample points $\{\vecy^\mmm\}_{\mmm=1}^{\MMM}$ in $\Gamma$. The basis usually consists of global fundamental Lagrange interpolating polynomials \cite{Babuska2007, Barthelmann2000, Nobile2008, Nobile2008a}. Then, in this case, the full approximation of $u(\bx,\vecy)$ with respect to both the spatial variable $\bx$ and parameter vector $\vecy$ takes the form of the interpolant
$$
u(\bx,\vecy)\approx {\mathcal I}_{\MMM} u_h(\bx,\vecy)  :=\sum_{\mmm=1}^\MMM u_h(\bx,\vecy^\mmm)\psi_{\mmm}(\vecy) \in V_\MMM(\Gamma)\otimes W_h(D) ,
$$
where $W_h(D)$ denotes, e.g., the finite-dimensional finite element space used for spatial approximation and $u_h(\bx,\vecy^\mmm)$ denotes the solution of the discretized PDE for the sample parameter vector $\vecy^m$. Here, we approximate the quantity of interest $\QQQ$ given in \eqref{eqn:intro_stat_qoi_3} by the quantity
\begin{equation}\label{eqn:q_estimate_sc1}
 \QQQ\approx \QQQ_{\MMM,h}^{\mathrm{SC}} := \int_\Gamma G({\mathcal I}_{\MMM} u_h) \rho(\vecy)\,d\vecy.
\end{equation}
In practice, this integral has to be further approximated. If the mapping $G({\mathcal I}_{\MMM} u_h):\Gamma \rightarrow \widetilde W(D)$ is sufficiently smooth, one can use an interpolatory quadrature rule for which the quadrature points $\{\vecy^\mmm\}_{\mmm=1}^{\MMM}$ and Lagrange fundamental polynomial basis functions $\{\psi_\mmm\}_{\mmm=1}^{\MMM}$ are the same as those used to define the interpolant ${\mathcal I}_{\MMM} u_h$. If $\{\mu_\mmm\}_{\mmm=1}^\MMM$ denotes the corresponding quadrature weights, we then have from \eqref{eqn:q_estimate_sc1} that
$$
\QQQ_{\MMM,h}^{\mathrm{SC}} \approx \sum_{\mmm=1}^{\MMM} \mu_\mmm G\big( {\mathcal I}_{\MMM} u_h(\bx,\vecy^m)\big)
=\sum_{\mmm=1}^{\MMM} \mu_{\mmm} G\Big( \sum_{\mmm'=1}^\MMM u_h(\bx,\vecy^{\mmm'})\psi_{\mmm'}(\vecy^m)\Big)= \sum_{\mmm=1}^{\MMM} \mu_\mmm G\big( u_h(\bx,\vecy^\mmm)\big),
$$
since the Lagrange fundamental polynomials satisfy $\psi_{\mmm'}(\vecy^m)=\delta_{\mmm\mmm'}$. In general, the numerical approximation of the integral in \eqref{eqn:q_estimate_sc1} can also be achieved using a different quadrature rule. The overall computational cost of this rule, however, is negligible compared to the cost of constructing the interpolant $\mathcal I_{\MMM}u_h$.

\vspace{1em}

Thus, comparing with \eqref{eqn:qoi_full_approx}, we see that SC methods for approximating the quantity of interest are sampling methods much in the same vein as are MC methods. For the former, the sample points $\{\vecy^\mmm\}_{\mmm=1}^{\MMM}$ and weights $\{\mu_\mmm\}_{\mmm=1}^\MMM$ in \eqref{eqn:qoi_full_approx} are chosen from an interpolatory quadrature rule whereas for the latter, they are chosen at random and with weights $1/\MMM$ for all $\mmm$. For both, the total computational effort is dominated by the computation of solutions of the discretized PDE at the sample points $\vecy^\mmm$.\footnote{Instead of the Lagrange fundamental polynomials, one can choose other bases such as those composed of piecewise polynomial splines \cite{Ma2009}.}


\vspace{1em}

In this paper, in the same way as for MLMC methods \cite{Cliffe2011,Teckentrup2012,Charrier2011}, we consider reducing the cost of determining approximations of quantities using a hierarchy of spatial grids but, instead of using MC approximations with respect to the random parameters $\vecy$, we use {\em sparse-grid stochastic collocation} methods \cite{Babuska2007, Barthelmann2000, Nobile2008, Nobile2008a}. These sampling methods, based on nodal interpolation at sparse-grid points in $\Gamma$, have been shown to yield considerably higher rates of convergence than Monte Carlo methods for integrands $G\big((u(\vecy)\big)$ that depend smoothly on the random vector $\vecy\in \Gamma$ and for a moderately high parameter dimension $\NNN$. Thus, our goal is to use a hierarchy of spatial grids to accelerate the convergence of stochastic collocation approximations $\QQQ_{\MMM,h}^\mathrm{SC}$ defined in\eqref{eqn:q_estimate_sc1}, i.e., we want to do for stochastic collocation methods what MLMC methods do for MC methods.
 


\vspace{1em}

In Section \ref{section:notation}, we establish the notation and describe the problem setting used throughout the paper. In Section \ref{section:efficiency_sampling}, the $\varepsilon$-cost for sparse grid stochastic collocation methods, a measure of the efficiency of a sampling scheme, is discussed as is its computation based on \emph{a priori} error estimates. We introduce multilevel methods in Section \ref{section:multilevel} and derive formulae for the optimal sample size at each spatial discretization level from the error estimates given in Section \ref{section:efficiency_sampling}. We also derive a theoretical bound on the $\varepsilon$-cost that improves upon that of traditional collocation methods. Here it is necessary to distinguish between collocation methods with sampling errors with algebraic convergence, i.e., of order $O(\MMM^{-\mu_2})$, and those with sub-algebraic convergence, i.e., of order $O\left(\MMM^{-\mu_2}\log(\MMM)^{\mu_1}\right)$. Current practice in multilevel algorithms is to choose the hierarchy of spatial discretizations based on a fixed, predetermined mesh refinement strategy. Numerical examples are provided in Section \ref{section:numerical_examples} to complement and illustrate the theoretical results.

\section{Notation and Setting}
\label{section:notation}

In this section, we introduce notation, establish estimates for the approximation error in \eqref{eqn:q_estimate_sc1}, and make assumptions that allow us to  analyze the multilevel sparse grid method. Although the multilevel framework is applicable to a variety of physical models, we use the elliptic partial differential equation throughout as an illustrative example. Not only is it the most well-understood model problem in the context of sparse grid stochastic collocation methods, but it has also been used extensively as an application for multilevel Monte Carlo methods, thus serving as a useful basis for comparison. In sequel, let $D\subset \R^d, d=1,2,3$ be a convex polyhedron, or have $C^2$ boundary $\partial D$. We denote by $L_\rho^q(\Gamma;W(D))$, $1 \leq q \leq \infty$, the space of $q$-integrable $W(D)$-valued functions on $\Gamma$. The stationary elliptic equation with homogenous Dirichlet boundary conditions, in which both the conductivity coefficient $a$ and the forcing term $f$ are finite noise random fields can be written as a parameterized family of deterministic equations 
\begin{equation}\label{eqn:elliptic_parametric_y}
\begin{split}
\nabla\cdot( a(x,\vecy)\nabla u(x,\vecy) & = f(x,\vecy) \ \ \text{in } D\times \Gamma\\
u(x,\vecy) &= 0 \ \ \ \ \ \ \ \ \text{on } \partial D\times \Gamma ,
\end{split}
\end{equation}
with corresponding weak form: find $u:\Gamma \rightarrow H^1_0(D)$ so that 
\begin{equation}\label{eqn:elliptic_weak_y}
\int_D a(\vecy) \nabla u \cdot \nabla w \;dx = \int_D f(\vecy) w \; dx \ \ \ \forall w\in H^1_0(D), y\in \Gamma.
\end{equation} 
Under the assumption that $f\in L_\rho^\infty(\Gamma;L^2(D))$ and $a\in L^\infty(\Gamma,C^1(\bar D))$ so that
\[
0<a_{\min} \leq a(x,\vecy) \ \ \text{a.s.\! on } \Gamma\times D
\]
for constant $a_{\min}>0$, the solution to \eqref{eqn:elliptic_weak_y} exists, is unique and has sample paths $u(\vecy)\in H^1_0(D)\cap H^2(D)$. In fact, there exists a constant $C_\mathrm{reg}>0$ independent of $\vecy$ so that $\|u(\vecy)\|_{H^2}\leq C_\mathrm{reg}\|f(\vecy)\|_{L^2}$ for all $\vecy\in \Gamma$ and hence $u\in L^\infty_\rho(\Gamma,H^1_0(D)\cap H^2(D))$. 

\vspace{1em}

Our goal is to derive an estimate for $\|\QQQ - \QQQ_{\MMM,h}\|_{\widetilde W}$. It is convenient to use the linearity of the expectation, together with the triangle inequality to split the total error into a spatial discretization error and a sampling error, i.e. 
\begin{equation}\label{eqn: split generic error}
\|\QQQ - \QQQ_{M,h}^{\mathrm{SC}}\|_{\widetilde W} \leq \underbrace{\left\|\QQQ - \QQQ_h\right\|_{\widetilde W}}_{spatial\ error} + \underbrace{\left\|\QQQ_{h} - \QQQ_{\MMM,h}^{SC}]\right\|_{\widetilde W}}_{sampling\ error},
\end{equation}
where $\|\cdot\|_{\widetilde W}$ is the norm on $\widetilde W(D)$. Here, the spatial discretization error is independent of the sampling error and can thus be considered separately. 
\subsection{Spatial discretization error}

We estimate the first term of the right-hand side of \eqref{eqn: split generic error}. With regards to the output of interest $G(u)$, we make the following assumptions.  

\begin{assumption}\label{ass:G regularity}
\mbox{(i)}
{\em  For each $\vecy\in\Gamma$, $u(\bx,\vecy)\in W(D)$ and $G\big(u(\bx,\vecy)\big)\in \widetilde W(D)$ for appropriate function spaces $W(D)$ and $\widetilde W(D)$. For second-order elliptic problems, often $W(D)=H^1(D)$ or a subspace of that Sobolev space; if $G(u)$ is a functional, then $\widetilde W(D)={\mathbb R}$.}

\mbox{(ii)}
{\em For all $u_1(\bx,\vecy),u_2(\bx,\vecy)\in W(D)$ and $\vecy\in\Gamma$, the mapping $G:W(D)\rightarrow \widetilde W(D)$ satisfies the Lipschitz condition
\begin{equation}\label{assum1}
\big\|G\big(u_1(\cdot,\vecy)\big)-G(u_2(\cdot,\vecy)\big)\big\|_{\widetilde W} \leq C_G(\vecy)\big\|u_1(\cdot,\vecy)-u_2(\cdot,\vecy)\big\|_{W}, 
\end{equation}
where the Lipschitz constant $C_G(\vecy)\in L_\rho^1(\Gamma)$.}
\hfill$\Box$
\end{assumption} 

The regularity assumption \eqref{assum1} together with the Jensen and H\"older inequalities yield that \begin{equation}\label{errest2}
\begin{aligned}
\|\QQQ - \QQQ_{h}\|_{\widetilde W} &= \big\|\E[G(u)-G(u_h)]\big\|_{\widetilde W} \leq \E[\|G(u)-G(u_h)\|_{\widetilde W}] \\ 
&\leq \E\left[C_G\|u-u_h\|_{W}\right] \leq \|C_{G}\|_{L_\rho^1(\Gamma)} \|u-u_h\|_{L_\rho^\infty(\Gamma,W)}. 
\end{aligned}
\end{equation}
The spatial error $\|u-u_h\|_{L^\infty(\Gamma,W)}$ can often be approximated by means of traditional finite element analyses; see, e.g., \cite{Brenner2007}. For second-order elliptic PDEs with homogeneous Dirichlet boundary conditions, under standard assumptions on the spatial domain $D$ and the data, one can choose $W(D)=H^s(D)$, $s=0$ or $1$, i.e., we can measure the error in either the $H^1(D)$ or $H^0(D)=L^2(D)$ norms. One can then construct $u_h(\cdot,\vecy)\in V_h(D)\subset H^1_0(D)$, where $V_h(D)$ denotes a standard finite element space of continuous piecewise polynomials of degree at most $r$ based on a regular triangulation $\mathcal T_h$ of the spatial domain $D$ with maximum mesh spacing parameter $h:=\max_{\tau\in \mathcal T_h}\textrm{diam}(\tau)$. We then have the error estimate \cite{Brenner2007}
\begin{equation}\label{eqn:fem_error_ptwise}
\|u(\cdot,\vecy)-u_h(\cdot,\vecy)\|_{H^s(D)}\leq ch^{r+1-s}\|u(\cdot,\vecy)\|_{H^{r+1}(D)} \ \ \text{for $s=0,1$ and for a.e. $\vecy\in \Gamma$},
\end{equation}
where $c > 0$ is independent of $\vecy$ and $h$. Hence,
\begin{equation}\label{eqn:fem_error_linf}
\|u-u_h\|_{L^\infty(\Gamma,H^s(D))}\leq c h^{r+1-s}\|u\|_{L^\infty(\Gamma,H^{r+1}(D))} \ \ \text{for $s=0,1$}.
\end{equation}
For finite element error estimates under less rigid conditions, see, e.g., \cite{Teckentrup2012, Grisvard1985}. Combining \eqref{errest2} and \eqref{eqn:fem_error_ptwise} yields
\begin{equation}\label{errest3}
 \big\|\E[G(u)-G(u_h)]\big\|_{\widetilde W} 
\leq c h^{r+1-s}\|C_{G}\|_{L_\rho^1(\Gamma)} \|u\|_{L^\infty(\Gamma,H^{r+1}(D))}. 
\end{equation}

\subsection{Sampling Error}

In light of Assumption \ref{ass:G regularity}, the sampling error in \eqref{eqn: split generic error} can be bounded as follows

\begin{align*}
\|\QQQ_h - \QQQ_{M,h}^{\mathrm{SC}}\|_{\widetilde W} &= \left\|\E\big[G(u_h) - G\big(\mathcal{I}_{\MMM}u_h\big)\big]\right\|_{\widetilde W} \leq \E\big[\|G(u_h)-G\big(\mathcal I_{\MMM}u_h\big)\|_{\widetilde W}\big]\\
&\leq \|C_G\|_{L^1_\rho} \|u_h - \mathcal I_{\MMM}u_h\|_{L^\infty(\Gamma,W)}
\end{align*}

It therefore suffices to consider only the error of interpolating finite element solutions $u_h$ in the stochastic variable $\vecy\in \Gamma$. In the following, we briefly outline the construction of sparse grid interpolants and elaborate on the resulting interpolation error estimates that we will make use of in the following sections. 

\vspace{1em}

Most $N$-dimensional interpolants are constructed through some combination of lower dimensional interpolants. For each component $\Gamma_n\subset \R$ of $\Gamma$, let
\[
V_{i_n}(\Gamma;W(D))=\left\{\sum_{j=1}^{m_{i_n}}c_j\psi_n^{j}: c_j \in W(D) \ \text{for } j=1,...,m_{i_n}\right\},
\]
where $\psi_n^1,...,\psi_n^{m_{i_n}}$ is a set of one-dimensional nodal basis functions with interpolation level $i_n$ and based on $m_{i_n}$ nodal points $y_n^1,...,y_n^{m_{i_n}}$. Furthermore, define $\mathscr U^{i_n}:C^0(\Gamma_n;W(D))\rightarrow V_{i_n}(\Gamma_n;W(D))$ to be the one-dimensional interpolation operator on $\Gamma_n$, so that for any one-dimensional function $u$ and any point $y_n \in \Gamma_n$, 
\[
\mathscr U^{i_n}(u)(y_n) = \sum_{j=1}^{m_{i_n}}u(y_n^j) \psi_{n}^{j}(y_n). 
\] 
The full tensor product interpolant of level $\nu$ approximates an $N$-dimensional function $u:\Gamma\rightarrow W(D)$ by the product of one-dimensional interpolants, each with interpolation level $i_n=\nu$, i.e.
\begin{equation}\label{eqn:full_tensor_prod}
u(\vecy) \approx \mathscr U^{\nu}\otimes \cdots \otimes \mathscr U^{\nu}(u)(\vecy) := \sum_{j_1=1}^\nu \cdots \sum_{j_N=1}^\nu u(y_1^{j_1},...,y_N^{j_N}) \prod_{n=1}^N \psi_{n}^{j_n}(y_n). 
\end{equation}
Computing this interpolant requires the evaluation of $v$ at $\MMM=\prod_{n=1}^{N}m_{i_n} = (m_\nu)^N$ sample points, leading to a prohibitively high cost at high values of $N$, especially if each function evaluation involves a PDE solve. 

\vspace{1em}

The isotropic Smolyak formula \cite{Smolyak1963} constructs a multi-dimensional interpolant $\mathcal I_{\MMM}u$ on $\Gamma$ from univariate interpolants, based on a greatly reduced set of sample points $\{\vecy^1,...,\vecy^{\MMM}\}$ while maintaining an overall accuracy not much lower than that of the full tensor product rule (see \cite{DeVore1998,Bungartz2004}). For any multi-index $i=(i_1,...,i_N)\in \N_+^N$, take $i\geq 1$ to mean $i_n\geq 1$ for $n=1,...,N$ and let $|i|:=i_1+...+i_N$. Also for any coordinate $y_n$ of $\vecy\in \Gamma$, we write $\vecy=(\vecy_n,y_n^*)$, where $y_n^*\in \prod_{\substack{n'=1\\n'\neq n}}^N \Gamma_{n'}$ are the remaining coordinates.  While not computed as such, the Smolyak interpolation operator $\mathcal I_{\MMM}$ of level $\nu$ can be written as the linear combination of tensor product rules
\[
\mathcal I_{\MMM} = \sum_{\substack{\nu-N+1\leq |i-1| \leq \nu\\ i\geq 1 }} (-1)^{\nu+N-|i|} {N-1\choose \nu + N -|i|}\mathscr U^{i_1}\otimes\cdots \otimes \mathscr U^{i_n}.
\] 
In the following, we restrict our attention to bounded hyper-rectangles $\Gamma$, assuming without loss of generality that $\Gamma=[-1,1]^N$, and consider the isotropic Smolyak formula based on one-dimensional Clenshaw-Curtis nodes 
\[
y_n^j = - \cos\left(\frac{\pi (j-1)}{m_{i_n}-1}\right),\ \text{for } j = 1,2,...,m_{i_n},
\] 
with $m_{i_n}$ chosen so that
\[
m_{i_n} = \left\{ \begin{array}{ll} 1, & \text{ if } i_n = 1 \\ 2^{i_n-1}+1, & \text{ if } i_n>1 \end{array} \right.
\]
to ensure nestedness. Extensions of the Smolyak formula to unbounded domains with non-nested Gaussian abscissas can be found in \cite{Nobile2008}, while \cite{Nobile2008a} discusses anisotropic Smolyak formulae in which coordinate directions can be weighted differently, according to their relative importance.

\vspace{1em}

For the purposes of error estimation for sparse grid methods, the integrand $u_h$ is often required to have bounded mixed derivatives of order $k\in\N_0$, i.e. to belong to the space%
\begin{equation*}
C_\mix^k(\Gamma,W(D)) = \left\{w:\Gamma\rightarrow W(D): \|w\|_{\mix,k}: =\max_{y\in\Gamma, s\leq k}\|D^s w(y)\|_{W} < \infty\right\},
\end{equation*} 
where $s=(s_1,...,s_N)$ is a multi-index in $\N_{0,+}^N$. 

\vspace{1em}

Conditions on the smoothness of the model output $u_h$ in $\vecy\in\Gamma$ depend on the underlying physical model and can often be related to the smoothness of the model's input parameters. For the elliptic problem \eqref{eqn:elliptic_parametric_y}, it was shown in \cite{Babuska2007} (Lemma 3.2) that if
\[ 
\|\partial^l_{y_n}a(y)\|_{L^\infty}\leq \theta_n, \ \ \|\partial^l_{y_n}f(y)\|_{L^2} \leq \theta_n, \ \ \text{a.e. on } \Gamma, \ \text{for all } l = 1,2,...,k \ \text{ and all } n = 1,\cdots, N, 
\]
where $0<\theta_n<\infty$ is independent of $\vecy = (\vecy_n,y_n^*)\in \Gamma$, then $u_h\in C_{\mix}^k(\Gamma,H^1(D))$.
The above condition is readily satisfied by standard finite noise approximations of the coefficients. In \cite{Barthelmann2000} (and later in \cite{Nobile2008}) it was shown that for functions in $C_{\mix}^k$, the interpolation error for the isotropic Smolyak approximation based on global Lagrange polynomials has upper bound of the form
\begin{equation}\label{eqn:sample_error_lagrange_mixed}
\|u-\mathcal I_{\MMM} u\|_{C^0(\Gamma,W)}\leq c \MMM^{-k}\log(\MMM)^{(k+2)(N-1)+1} \|u\|_{\mix,k}.
\end{equation}
The works \cite{Bungartz2004,Ma2009} make use of piecewise linear nodal basis functions with local support to interpolate functions with limited smoothness, obtaining an estimate on the sampling error for functions in $C_{\mix}^2(\Gamma;W(D))$ of the form, 
\begin{equation}\label{eqn:sample_error_pw_linear}
\|u-\mathcal I_{\MMM}u\|_{C^0(\Gamma,W)}\leq c \MMM^{-2}\log(\MMM)^{3(N-1)} \|u\|_{\mix,2}.
\end{equation}
The hierarchical construction of the piecewise linear sparse grid interpolant also lends itself well to adaptive refinement through the use the hierarchical surplus as an indicator of discontinuity. This approach has been extended to constructions using wavelets (see \cite{Gunzburger2013}).

\vspace{1em}

The convergence rate in \eqref{eqn:sample_error_lagrange_mixed} was improved in \cite{Nobile2008} to an algebraic rate for integrands within a special class of functions $C_{\mix}^\infty(\Gamma,W(D))$ that have analytic extension in each direction. In particular, $u\in C^0(\Gamma,W(D))$ is a member of $C_{\mix}^{\infty}(\Gamma;W(D))$ if for every $y = (y_n,y_n^*)\in \Gamma, n=1,...,N$, the function $u(y_n,y_n^*,x)$ as a univariate function of $y_n$, i.e. $u:\Gamma_n \rightarrow C^0(\Gamma_n^*,W(D))$, admits an analytic extension $u(z), z\in \mathbb C$ in the complex region
\[
\Sigma(\Gamma_n;\tau_n):\{z\in \mathbb C: \mathrm{dist}(z,\Gamma_n)\leq \tau_n\},
\]
so that 
\[
|u|^{(n)}_{\mix,\infty}:=\max_{z\in \Sigma(\Gamma_n;\tau_n)} \|u(z)\|_{C^0(\Gamma_n^*;W)} < \infty.
\]
Let 
\[
\|u\|_{\mix,\infty}:= \max_{n=1,...,N} |u|^{(n)}_{\mix,\infty}.
\]
For the elliptic equation \eqref{eqn:elliptic_parametric_y}, the following mild assumption on coefficients $a$ and $f$ guarantees that $u_h\in C_{\mix}^\infty(\Gamma,H^1(D))$ (see \cite{Babuska2007}, Lemma 3.2).%
\begin{assumption}\label{ass:coefficient_regularity}
Assume that for every $y=(y_n,y_n^*)\in \Gamma$, there is a constant  $\theta_n<\infty$ so that%
\begin{equation}\label{eqn:coefficient_regularity}
\left\|\frac{\partial_{y_n}^k a(y)}{a(y)}\right\|_{L^\infty}\leq \theta_n^k k! \ \ \ \text{ and } \ \ \ \frac{\|\partial_{y_n}^k f(y)\|_{L^2}}{1+\|f(y)\|_{L^2}}\leq \theta_n^k k!,
\end{equation}
for all $k\in \N_0^+$.
\end{assumption}
Although the sampling error estimates derived in \cite{Nobile2008} depend on the norms $|u|_{\mix,\infty}^{(n)}$, where $n=1,...,N$, these were subsumed into a scaling constant. For our purposes, however, it is necessary for them to appear explicitly in the error estimate. The following lemma therefore indicates how the derivations in \cite{Nobile2008} can be modified to achieve this. 
\begin{lemma}
Let $\mathscr A(\nu,N)u$ be the Smolyak interpolant of the function $u$ contained in $C_{\mix}^{\infty}(\Gamma,W(D))$, based on Clenshaw-Curtis abscissas and Lagrange polynomials. The interpolation error then satisfies
\begin{equation}\label{eqn:sampling_error_analytic}
\|u-\mathcal I_{\MMM} u\|_{C^0(\Gamma,W)}\leq c\MMM^{-\mu_2}\max\{\|u\|_{\mix,\infty},\|u\|_{\mix,\infty}^{N}\}, 
\end{equation}
for constants $c \geq 1$ and $\mu_2 > 0$.
\end{lemma}
\begin{proof}
The estimation of the interpolation error of $u$ over the domain $\Gamma\subset \R^N$ is based on its one-dimensional counterparts. Indeed it was shown in \cite{Nobile2008} (see also  \cite{Babuska2007}, Lemma 4.4) that for functions $u$ in $C_{\mix}^{\infty}(\Gamma;W(D))$, 
\[
\|u-\mathscr U^{(i_n)}u\|_{C^0(\Gamma_n;W(D))}\leq  C i_n e^{-\sigma 2^{i_n}},
\]
where $\displaystyle\sigma=\max_{n=1,...,N}\frac{1}{2}\log\left(\frac{2\tau_n}{|\Gamma_n|}+\sqrt{1+\frac{4\tau_n^2}{|\Gamma_n|^2}}\right)$, and $C = \frac{4(\pi+1)e^{2\sigma}}{\pi(e^{2\sigma}-1)}\|u\|_{\mix,\infty}= \widetilde C \|u\|_{\mix,\infty}$. 
Lemma 3.3 in \cite{Nobile2008} then uses these estimates to bound the Smolyak interpolation by
\begin{align}
&\|u - \mathcal I_{\MMM}u\|_{C^0(\Gamma;W(D))}\leq \frac{1}{2}\sum_{n=1}^N (2 C)^n\sum_{\substack{i\geq 1\\ |i-1|=\nu}} \left(\prod_{l=1}^n i_l \right) e^{-\sigma \sum_{l=1}^n 2^{i_l-1}}\nonumber\\
\leq &\max\left\{\|u\|_{\mix,\infty},\|u\|_{\mix,\infty}^{N}\right\} \frac{1}{2}\sum_{n=1}^N (2 \widetilde C)^n\sum_{\substack{i\geq 1\\ |i-1|=\nu}} \left(\prod_{l=1}^n i_l \right) e^{-\sigma \sum_{l=1}^n 2^{i_l-1}}.
\end{align}
The remainder of the derivation in \cite{Nobile2008} (Lemma 3.4, and Theorems 3.6 and 3.9) remains unchanged, except for the replacement of the constant $C$ in with $\widetilde C$ and the addition of the term $\max\left\{\|u\|_{\mix,\infty},\|u\|_{\mix,\infty}^{N}\right\}$. Theorem 3.9 in \cite{Nobile2008} then asserts
\[
\|u - \mathcal I_M u\|_{C^0(\Gamma;W(D))}\leq c M^{-\mu_2}\max\{\|u\|_{\mix,\infty},\|u\|_{\mix,\infty}^{N}\}, 
\]
where
\[
c = \frac{C_1(\sigma,\delta^*)e^{\sigma}}{|1-C_1(\sigma,\delta^*)|}\max\{1,C_1(\sigma,\delta^*)\}^N, \ \ \ \mu_2 = \frac{\sigma}{1+\log(N)}, \ \text{and}  
\]
$C_1(\sigma,\delta^*)$ is defined in \cite{Nobile2008}, Equation (3.12).
\end{proof}
In summary, the sampling error estimates \eqref{eqn:sample_error_lagrange_mixed}, \eqref{eqn:sample_error_pw_linear} and \eqref{eqn:sampling_error_analytic} discussed in this section can therefore  all be written in the form%
\begin{equation}\label{ass:sample_error}
\|u - \mathcal I_M u\|_{W} \leq c_3 \log(\MMM)^{\mu_1}\MMM^{-\mu_2} \varphi(u_h),
\end{equation}
where $c_3\geq 1, \mu_1\geq 0,$ and $\mu_2>0$ and $\varphi:W(D)\rightarrow [0,\infty)$ satisfies $\varphi(u_n)\rightarrow 0$ for any sequence $u_n \rightarrow 0$ in $C_{\mix}^k(\Gamma;W(D))$ for $k\in \N\cup\{\infty\}$.

\section{The Efficiency of Sampling Methods}\label{section:efficiency_sampling}

\vspace{1em}
A useful indicator of an algorithm's efficiency is its $\varepsilon$-cost $\C_\varepsilon$, defined as the amount of computational effort required to reach a given level of accuracy $\varepsilon>0$. This effort can be measured in terms of the number of floating point operations or CPU time and is estimated based on \textit{a priori} error estimates. We now proceed to estimate the $\varepsilon$-cost of the sampling schemes discussed above. In general, the total cost $\C(\Q_{\MMM,h})$ of computing the estimate $\Q_{\MMM,h}$ is approximately
\[
\C(\QQQ_{\MMM,h}) = \sum_{m=1}^{\MMM} \C_h^{(m)},
\]
where $\C_{h}^{(m)}$ is the cost of computing the $m^{th}$ sample at spatial refinement level $h$. If the cost of a system solve is the same for all sample paths, i.e. $\C_h^{(m)}=\C_h$ for $m=1,...,\MMM$ then this sum simplifies to 
\begin{equation}\label{eqn:total_cost_simple}
\C(\QQQ_{\MMM,h})= M \C_h.
\end{equation}
Sampling methods are fully parallelizable and the cost savings of a parallel implementation can be readily incorporated into this cost estimate. Indeed, if the stochastic simulation is distributed among $N_\mathrm{batch}$ processors then the total cost is simply scaled by $\frac{1}{N_{\mathrm{batch}}}$.
In addition, we assume here that $\C_h$ grows polynomially with decreasing spatial refinement level $h$, i.e. there are constants $h_0 > 0$, $c_2 \geq 1$ and $\gamma > 0$, so that. 
\begin{equation}\label{ass:cost_per_sample}\tag{A2}
c_2 h^{-\gamma} \leq \C_h  \ \text{for all }  0 < h < h_0.
\end{equation}
The $\varepsilon$-cost for a sampling method can then be bounded by determining the lowest values of $h$ and $\MMM$ for which both the spatial error and the sampling error are less than $\frac{\varepsilon}{2}$, and substituting these values into \eqref{eqn:total_cost_simple}, using \eqref{ass:cost_per_sample}. Indeed, supposing the spatial disretization error has upper bound of the form $\|u-u_h\|_{L^\infty(\Gamma,W)}\leq c_1 h^\alpha$ for some $c_1\geq 1, \alpha>0$, then $h<\frac{1}{2c_1}\varepsilon^{\frac{1}{\alpha}}$ ensures that the spatial refinement error is within the tolerance level $\frac{\varepsilon}{2}$, and hence
\[
\C_h \geq c_2(2c_1)^\gamma \varepsilon^{-\frac{\gamma}{\alpha}}.
\]
If the upper bound in the generic sparse grid sampling error \eqref{ass:sample_error} doesn't contain a logarithmic term, i.e. if $\mu_1=0$, then it readily follows that a sample size $\MMM \geq (2c_3 \varphi(u_h))^{\frac{1}{\mu_2}}\varepsilon^{-\frac{1}{\mu_2}}$ guarantees 
a sampling error within the tolerance level $\frac{\varepsilon}{2}$. In this case, the $\varepsilon$-cost is at least
\begin{equation}\label{eqn:epsilon_cost_simple}
\C_{\varepsilon}(\QQQ_{\MMM,h}) = \MMM \C_h \geq (2c_3 \varphi(u_h))^{\frac{1}{\mu_2}} c_2(2c_1)^\gamma \varepsilon^{-\frac{1}{\mu_2}-\frac{\gamma}{\alpha}} = O(\varepsilon^{-\frac{1}{\mu_2}-\frac{\gamma}{\alpha}}).  
\end{equation}
We assume here implicitly that the term $\varphi(u_h)$ remains more or less unchanged as $h\rightarrow 0^+$, a reasonable assumption if $u_h \rightarrow u$. For the general case when $\mu_1>0$, the minimal sample size required $\MMM$ is slightly more involved. We derive such values in the following lemma. Note that for any $x\in \mathbb{R}$, $\lceil x \rceil$ denotes the unique integer $n$, so that $x\leq n < x + 1$.  

\begin{lemma}\label{lemma: epsilon cost for logs}
Let $0<\mu_2,\tilde\mu_2$ and $0<\mu_1\leq\tilde \mu_1$ be constants and suppose $0<\varepsilon<1$. 
If 
\begin{equation}\label{eqn: log error tol sample size}
\MMM = \left\lceil\varepsilon^{-\frac{1}{\tilde \mu_2}}\log\left(\varepsilon^{-1}\right)^{\frac{\tilde \mu_1}{\mu_2}}\right\rceil,
\end{equation}
then 
\begin{equation}
\MMM^{-\mu_2}\log\left(\MMM\right)^{\mu_1} \leq \left(1+\frac{\tilde \mu_1}{\mu_2}+\frac{1}{\tilde\mu_2}\right)^{\mu_1} \varepsilon^{\frac{\mu_2}{\tilde \mu_2}}
\end{equation}
\end{lemma}
\begin{proof}
The definition of the $\lceil \cdot \rceil$ operation implies 
\[
\varepsilon^{-\frac{1}{\tilde\mu_2}}\log\left(\varepsilon^{-1}\right)^{\frac{\tilde\mu_1}{\mu_2}} \leq \MMM < \varepsilon^{-\frac{1}{\tilde\mu_2}}\log\left(\varepsilon^{-1}\right)^{\frac{\tilde\mu_1}{\mu_2}} + 1
\]
and hence
\begin{equation}\label{eqn: lemma eta bound}
\MMM^{-\mu_2} \leq \left( \varepsilon^{-\frac{1}{\tilde\mu_2}}\log\left(\varepsilon^{-1}\right) ^{\frac{\tilde\mu_1}{\mu_2}} \right)^{-\mu_2} = \varepsilon^{\frac{\mu_2}{\tilde\mu_2}} \log\left(\varepsilon^{-1}\right)^{-\tilde\mu_1}, 
\end{equation}
Moreover, using the inequality $\log(x)< \frac{x^s}{s}$ for all $x,s>0$ and the fact that $\varepsilon<1$, we get
\begin{align}
\log(\MMM)^{\mu_1}%
&< \log\left(\varepsilon^{-\frac{1}{\tilde\mu_2}}\log\left(\varepsilon^{-1}\right) ^{\frac{\tilde\mu_1}{\mu_2}} + 1\right)^{\mu_1} 
\leq \log\left(\varepsilon^{-\left(\frac{1}{\tilde\mu_2}+\frac{\tilde\mu_1}{\mu_2}\right)} + 1\right)^{\mu_1}  \nonumber\\
&< \log\left(\varepsilon^{-\left(\frac{1}{\tilde\mu_2}+\frac{\tilde\mu_1}
{\mu_2}\right)}+(e-1)\varepsilon ^{-\left(\frac{1}{\tilde\mu_2}+\frac{\tilde\mu_1}{\mu_2}\right)}\right)^{\mu_1}  \nonumber \\
&= \left(1 + \left(\frac{1}{\tilde\mu_2}+\frac{\tilde\mu_1}{\mu_2}\right)\log(\varepsilon^{-1})\right) ^{\mu_1}.\label{eqn: lemma log(eta) bound}
\end{align}
Combining inequalities \eqref{eqn: lemma eta bound} and \eqref{eqn: lemma log(eta) bound}  yields
\begin{align*}
\MMM^{-\mu_2}\log(\MMM)^{\mu_1} & \leq\log\left(\varepsilon^{-1}\right) ^{-\tilde\mu_1}\left(1+\left(\frac{1}{\tilde\mu_2} + \frac{\tilde\mu_1}{\mu_2}\right)
\log\left(\varepsilon^{-1}\right)\right)^{\mu_1} \varepsilon^{\frac{\mu_2}{\tilde\mu_2}}\\
&= \log\left(\varepsilon^{-1}\right) ^{-(\tilde\mu_1-\mu_1)}\left(\frac{1}{\log\left(\varepsilon^{-1}\right)} + \left(\frac{1}{\tilde\mu_2}+\frac{\tilde\mu_1}{\mu_2}\right)\right)^{\mu_1}\varepsilon^{\frac{\mu_2}{\tilde\mu_2}}\\
&\leq \left(1 + \left(\frac{1}{\tilde\mu_2}+\frac{\tilde\mu_1}{\mu_2}\right)\right) ^{\mu_1}\varepsilon^{\frac{\mu_2}{\tilde\mu_2}}.
\end{align*}
\end{proof}
\begin{remark}
By replacing $\varepsilon$ in formula \eqref{eqn: log error tol sample size} with 
\begin{equation}\label{eqn: smaller epsilon}
\tilde\varepsilon := \left(1 + \left(\frac{1}{\tilde\mu_2}+\frac{\tilde\mu_1}{\mu_2}\right)\right) ^{-\frac{\mu_1\tilde\mu_2}{\mu_2}}\varepsilon < \varepsilon < 1,
\end{equation}
we can in fact achieve the upper bound 
\[
\MMM^{-\mu_2}\log(\MMM)^{\mu_1} \leq \varepsilon.
\]
\end{remark}
The sample size $\MMM$ necessary to compute the $\varepsilon$-cost when $\mu_1>0$ is therefore of the order
\[
\MMM = O\left( \varepsilon^{-\frac{1}{\mu_2}}\log(\varepsilon^{-1})^{\frac{\mu_1}{\mu_2}}\right), 
\] 
leading to the $\varepsilon$-cost
\begin{equation}
\C_\varepsilon(\Q_{\MMM,h}) = O\left(\varepsilon^{-\frac{1}{\mu_2}-\frac{\gamma}{\alpha}}\log(\varepsilon^{-1})^{\frac{\mu_1}{\mu_2}}\right).
\end{equation}

\section{Multilevel Sampling}\label{section:multilevel}

\subsection{The Multilevel Algorithm}

Let $\{h_{\ell}\}_{\ell=0}^L$ be a sequence of spatial discretization parameters giving an increasing level of accuracy and let $h_L$ be chosen to ensure that the spatial error term in \eqref{eqn: split generic error} satisfies
\[
\|\QQQ - \QQQ_{h_L}\|_{\widetilde W} \leq \frac{\varepsilon}{2}. 
\]
Multilevel quadrature methods are based on an expansion of this fine scale approximation $G(u_{h_L})$ as the sum of an initial coarse scale approximation and a series of correction terms, i.e.
\[
G(u_{h_L}) = G(u_{h_0}) + \sum_{\ell=1}^L \left( G(u_{h_\ell}) - G(u_{h_{\ell-1}}) \right).
\]
Taking expectations on both sides yields
\[
\QQQ_{h_L} = \QQQ_{h_0} + \sum_{\ell=1}^L (\QQQ_{h_\ell} - \QQQ_{h_{\ell-1}})
\]
We now further estimate $\QQQ_{h_L}$ by approximating both the coarse approximation and each correction term in the above sum using a different interpolant, i.e.
\begin{equation}\label{eqn:multilevel estimate}
\QQQ \approx \QQQ_{\{M_\ell\},\{h_\ell\}}^{\mathrm{MLSC}} := \QQQ_{M_0,h_0}^{\mathrm{SC}} + \sum_{\ell = 1}^L \left(\QQQ_{M_\ell,h_\ell}^{\mathrm{SC}} - \QQQ_{M_\ell,h_{\ell-1}}^{\mathrm{SC}}\right). 
\end{equation}
Since the stochastic interpolation levels, i.e. the sample sizes $M_\ell$ can be chosen separately for each spatial refinement level, using the multilevel estimate gives us the flexibility to coordinate the sample sizes $M_0,...,M_L$ in such a way that more samples are drawn at coarse spatial refinement levels while samples at finer spatial refinement levels are sampled more sparingly, hopefully improving the computational efficiency. 
For the sake of comparison, we refer to the sampling methods discussed in the previous section as single level sampling methods, since only spatial discretizations at the highest refinement level $h_L$ are sampled. 

\vspace{1em}

The total error for the multilevel estimate can now be decomposed as follows
\begin{align}
&\left\|\QQQ -\QQQ_{\{M_\ell\},\{h_\ell\}}^{\mathrm{MLSC}} \right\|_{\widetilde W} \nonumber\\
\leq\ & \underbrace{\left\|\QQQ - \QQQ_{h_L}\right\|_{\widetilde W}}_{spatial\ error} + \underbrace{\left\|\QQQ_{h_0}-\QQQ_{M_0,h_0}^{\mathrm{SC}}\right\|_{\widetilde W} + \sum_{\ell = 1}^{L}  \left\|\left(\QQQ_{h_\ell} - \QQQ_{h_{\ell-1}}\right) - \left(\QQQ_{M_\ell,h_\ell}^{\mathrm{SC}} -\QQQ_{M_\ell, h_{\ell-1}}^{\mathrm{SC}}\right)\right\|_{\widetilde W}}_{multilevel\ sampling\ error}\label{eqn:multilevel_error}.
\end{align}
Just as in the total approximation error \eqref{eqn: split generic error} for single level sampling methods, the error in \eqref{eqn:multilevel_error} can thus be decomposed into a spatial discretization error, depending only on $h_L$ and a multilevel sampling error, quantifying the accuracy with which of the correction terms $G(u_{h_\ell})-G(u_{h_{\ell-1}})$ are approximated through interpolation. 

\vspace{1em}

The basic multilevel sampling method, based on numerical estimates $e_L^{\mathrm{space}}$ and $e_L^{\mathrm{sample}}$ of the spatial error and the multilevel sampling error respectively, is outlined in Algorithm \eqref{alg:basic_ml}.

\vspace{1em}

 \begin{algorithm}[H]
 \SetAlgoLined
 \SetKwInOut{Input}{Input}
 \SetKwInOut{Output}{Output}
 \Input{Tolerance level $\varepsilon > 0$, initial discretization level $h_0$}
 \Output{Maximum refinement level $L$, multilevel estimate $\QQQ_{\{\MMM_\ell\},\{h_\ell\}}^{\mathrm{MLSC}}$ of $\QQQ$}
 \BlankLine
 Determine initial sample size $\MMM_0$\;
 Generate sample $\left\{u_{h_0}(\bx,\vecy^m)\right\}_{m=1}^{\MMM_0}$ and compute $\Q_{\{\MMM_0\},\{h_0\}}^{\mathrm{MLSC}}=\QQQ_{M_0,h_0}^{\mathrm{SC}}$\;
 Set spatial error estimate $e_0^{\mathrm{space}}=1$, maximum refinement level $L = 0$\;
 \While{$e_L^{\mathrm{space}} > \frac{\varepsilon}{2}$}
 {
 $L \leftarrow L + 1$ \;
 Refine the model at new discretization level $h_L$\;\label{alg1_refine}
 Determine $\{\MMM_0,...,\MMM_L\}$ so that $ e_{L}^{\mathrm{space}} + e_{L}^{\mathrm{sample}}< \varepsilon$ while minimizing the total computational cost $\C\left(\Q_{\{\MMM_\ell\},\{h_\ell\}}^{\mathrm{MLSC}}\right)$\;\label{alg1_optimization} 
 Generate the samples $\left\{u_{h_\ell}(\vecy^{m})\right\}_{m=1}^{\MMM_\ell}$ for $\ell=0,...,L$\;\label{alg1_correction_samples}
 Update the multilevel estimate $\Q_{\{\MMM_\ell\},\{h_{\ell}\}}^{\mathrm{MLSC}}$\;\label{alg1:ml_update}
 Compute $e_L^{\mathrm{space}}$\;\label{alg1:error_estimation}
 }
 \caption{Basic multilevel sampling algorithm}\label{alg:basic_ml}
 \end{algorithm}

\vspace{1em}

We elaborate on some of the lines in Algorithm \ref{alg:basic_ml}, and outline some of the outstanding issues addressed in the remainder of this paper.
Traditionally (see \cite{Giles2008, Cliffe2011, Charrier2011}), the spatial grid refinement step \ref{alg1_refine} is achieved by scaling the mesh spacing parameter by a fixed percentage, i.e. $h_{L+1} = s h_L$ for $L=1,2,...$ and $0 < s < 1$. While this construction is convenient to analyze, it is not necessary for the convergence of the algorithm. In fact, the determination of adaptive mesh refinement strategies in this context is a topic of ongoing research.

\vspace{1em}

In some cases the integrand $G(\mathcal I_M u_h)$ is a spatially varying function, defined on some spatial mesh $\mathcal T_{h}$. The computation of the sample correction paths $\displaystyle G\big(\mathcal I_{M_\ell}u_{h_\ell}(\vecy^m)\big) - G\big(\mathcal I_{M_\ell}u_{h_{\ell-1}}(\vecy^m)\big)$ (line \ref{alg1_correction_samples}) that are used to update the multilevel estimate (see line \eqref{alg1:ml_update}), requires the spatial interpolation of $v_{h_\ell-1}^{(i)}$ at points on the refined mesh $\mathcal T_{h_\ell}$. In \cite{Barth2010}, this additional cost is mitigated through the use of hierarchical finite elements \cite{Yserentant1986}. For general spatial domains $D$, such hierarchical approximations are however not always tractable. 

\vspace{1em}

One benefit of using nested grids, such as the Clenshaw-Curtis sparse grid, is that the interpolant $\mathcal I_{M_{\ell-1}}u_{h_{\ell-1}}$, computed as the `fine' spatial grid interpolant of the previous correction term, can be used to construct the `coarse' spatial grid interpolant $I_{M_\ell}u_{h_{\ell-1}}$ of the next correction term. In fact if $M_{\ell-1} > M_{\ell}$, which is likely to be the case for the optimal sample sizes, no additional sample paths need to be generated. In contrast, Monte Carlo sampling requires sample paths of correction terms to be independent, which prohibits the re-use of sample paths. 

\vspace{1em}

Similar to single level methods, the total cost of computing the multilevel estimate \eqref{eqn:multilevel estimate} is dominated by the construction of the interpolants, i.e.
\begin{equation}\label{eqn:sg_total_cost}
\C\left( \QQQ_{\{M_\ell\},\{h_\ell\}}^{\mathrm{MLSC}}\right) \approx \sum_{\ell = 0}^L M_\ell C_\ell, 
\end{equation}
where $C_\ell$ is the combined cost of computing the sample paths of $u_{h_\ell}(\vecy^m)$ and $u_{h_{\ell-1}}(\vecy^{m})$ for each $m = 1,...,M_\ell$.

\vspace{1em}

\subsection{The Optimal Allocation Sub-Problem}

The determination of optimal sample sizes $\{\MMM_0,...,\MMM_L\}$ in \eqref{eqn:multilevel estimate} represents the most important step of Algorithm \ref{alg:basic_ml} and can be succinctly formulated as a discrete constrained optimization problem in $L+1$ variables. Since the spatial error is independent of the sample size, this term can be ignored. The sample sizes $\MMM_0,...,\MMM_L$ should then be chosen so as to minimize the total computational effort, while maintaining a sample error that is within. For convenience, we require both the sampling- and spatial errors to be bounded above by $\varepsilon/2$. Written as an optimization problem, line \ref{alg1_optimization} amounts to
\begin{equation}\label{eqn:optimization_general}
\begin{split}
& \min_{\MMM_0,...,\MMM_L} \sum_{\ell = 0}^L M_\ell C_\ell, \\ 
\text{subject to}\ &\  \left\|\QQQ_{h_0}-\QQQ_{M_0,h_0}^{\mathrm{SC}}\right\|_{\widetilde W} + \sum_{\ell = 1}^{L}  \left\|\left(\QQQ_{h_\ell} - \QQQ_{h_{\ell-1}}\right) - \left(\QQQ_{M_\ell,h_\ell}^{\mathrm{SC}} -\QQQ_{M_\ell, h_{\ell-1}}^{\mathrm{SC}}\right)\right\|_{\widetilde W} \leq \frac{\varepsilon}{2}.
\end{split}
\end{equation} 
Like the single-level sampling methods, the multi-level Algorithm \ref{alg:basic_ml} is amenable to parallel implementation, the effect of which can be incorporated into the total cost by simply dividing throughout by the batch size $N_\mathrm{batch}$. Since the inclusion of this factor does not change the optimization problem \eqref{eqn:optimization_general}, we leave it out for simplicity.

\vspace{1em}

As a matter of notational convenience, we define 
\[
\triangle u_\ell := \left\{ \begin{array}{ll}
u_{h_0} & \text{ for } \ell = 0\\
u_{h_\ell} - u_{\ell-1} &\text{ for } \ell = 1,2,...,L 
\end{array}\right.
\]
We want to bound the multilevel sampling error in \eqref{eqn:multilevel_error} by an expression involving the interpolation error $\|\triangle u_\ell - \mathcal I_M \triangle u_\ell\|_{C^0(\Gamma;W)}$ for which we have \emph{a priori} error estimates, such as \eqref{eqn:sample_error_lagrange_mixed},\eqref{eqn:sample_error_pw_linear}, or \eqref{eqn:sampling_error_analytic}. For the coarsest refinement level, $\ell =0$, this can achieved using Jensen's inequality together with Assumption \ref{ass:G regularity}, yielding
\begin{align*}
\left\| \QQQ_{h_0} - \QQQ_{M_0,h_0}^{\mathrm{SC}}\right\|_{\widetilde W} &= \left\|\E[G(u_{h_0})-G(\mathcal I_{M_0} u_{h_0})]\right\|_{\widetilde W}\leq \E\left[\|G(u_{h_0})-G(\mathcal I_{M_0}u_{h_0})\|_{\widetilde W}\right]\\
&\leq \|C_G\|_{L^1_\rho}\|u_{h_0} - \mathcal I_{M_0}u_{h_0}\|_{C^0(\Gamma;W)}\leq \tilde c_3 \log(M_0)^{\mu_1}M_0^{-\mu_2}\varphi(u_{h_0}).
\end{align*}
for the appropriate constant $\tilde c_3$. To ensure that similar upper bounds hold for the higher order correction terms, we make the following assumption. 

\begin{assumption}\label{ass:G frechet}
Assume that the mapping $G:W(D) \rightarrow \widetilde W(D)$ is continuously Fr\'echet differentiable.
\end{assumption}

Note that we try to remain agnostic regarding the smoothness of $G(u_h)$ with respect to the vector $\vecy\in \Gamma$, allowing the estimation of the integral $\int_\Gamma G(\mathcal I_{M} u_h)\rho(\vecy)\;d\vecy$ to be treated separately from the interpolation $\mathcal I_M u_h$ of $u_h$. Recall that  $\QQQ_{M,h}:=\int_\Gamma G(\mathcal I_M u_h)\rho(\vecy)\;d\vecy$.

\begin{lemma}
Suppose $u\in C^0(\Gamma,W)$ satisfies \eqref{eqn:elliptic_weak_y}, Assumption \ref{ass:G frechet} holds, and $Q$ is estimated by the multilevel estimate \eqref{eqn:multilevel estimate}. Then there exist constants $C_{G'}, C_{G''}>0$ such that for $\ell=1,2,...,L$
\begin{equation}\label{eqn:bound_mlv_by_mlu}
\begin{split}
\left\|\left(\QQQ_{h_\ell} - \QQQ_{h_{\ell-1}}\right) - \left(\QQQ_{M_\ell,h_\ell}^{\mathrm{SC}} -\QQQ_{M_\ell, h_{\ell-1}}^{\mathrm{SC}}\right)\right\|_{\widetilde W} \leq & \phantom{+} ( C_{G'} + C_{G''}\|\triangle u_\ell\|_{C^0(\Gamma,W)} ) \|\triangle u_\ell - \mathcal I_{M_\ell} \triangle u_\ell\|_{C^0(\Gamma,W)}\\
 & + C_{G''}  \|\triangle u_\ell\|_{C^0(\Gamma,W)} \|u_{h_{\ell-1}}-\mathcal I_{M_\ell} u_{h_{\ell-1}}\|_{C^0(\Gamma,W)}
\end{split}
\end{equation}

\end{lemma}
\begin{proof}

For spatial refinement levels $\ell\geq 1$, we use Jensen's inequality to obtain
\begin{align*}
& \left\|\E\big[G(u_{h_\ell})-G(u_{h_{\ell-1}}\big]-\E\big[G(\mathcal I_{M_\ell} u_{h_\ell})-G(\mathcal I_{M_\ell} u_{h_{\ell-1}})\big]\right\|_{\widetilde W}\\
\leq\ & \E\left[\left\|G(u_{h_\ell})-G(u_{h_{\ell-1}})-G(\mathcal I_{M_\ell} u_{h_\ell})-G(\mathcal I_{M_\ell} u_{h_{\ell-1}}])\right\|_{\widetilde W}\right]\\
\leq\  &\left\|G(u_{h_\ell})-G(u_{h_{\ell-1}})-G(\mathcal I_{M_\ell} u_{h_\ell})-G(\mathcal I_{M_\ell} u_{h_{\ell-1}}])\right\|_{C^0(\Gamma,\widetilde W)}. 
\end{align*}
For any fixed $\vecy\in\Gamma$, we now let $\triangle u_\ell = u_{h_\ell} - u_{h_{\ell-1}}$ make use of Taylor's Theorem for Banach spaces and the linearity of $\mathcal I_M$ to obtain
\begin{align*}
& G(u_{h_\ell})-G(u_{h_{\ell-1}})-\left(G(\mathcal I_{M_\ell} u_{h_\ell})-G(\mathcal I_{M_\ell} u_{h_{\ell-1}})\right)\\
=& \int_0^1 G'(u_{\ell-1} + t\triangle u_\ell)\triangle u_\ell dt  - \int_0^1 G'(\mathcal I_{M_\ell} u_{h_{\ell-1}} + t \mathcal I_{M_\ell} \triangle u_\ell)\mathcal I_{M_\ell} \triangle u_\ell dt\\
=& \left(\int_0^1 G'(u_{\ell-1} + t\triangle u_\ell)-G'(\mathcal I_{M_\ell}(u_{h_{\ell-1}} + t \triangle u_\ell)) dt\right) \triangle u_\ell \\
& \phantom{=} - \left(\int_0^1 G'(\mathcal I_{M_\ell}  u_{h_{\ell-1}} + t  \mathcal I_{M_\ell}\triangle u_\ell)dt\right)  (\triangle u_\ell - \mathcal I_{M_\ell} \triangle u_\ell)
\end{align*}
The first term can be further simplified through
\begin{align*}
& \left\|\left(\int_0^1 G'(u_{h_{\ell-1}} + t\triangle u_\ell )- G'(\mathcal I_{M_\ell}(u_{h_{\ell-1}} + t  \triangle u_\ell)) dt\right) \triangle u_\ell \right\|_{\widetilde W}\\
= & \left\| \int_0^1\int_0^1 G''( \xi(t,s) )ds\;(u_{h_{\ell-1}}-\mathcal I_{M_\ell} u_{h_{\ell-1}} + t(\triangle u_\ell - \mathcal I_{M_\ell}\triangle u_\ell))dt (\triangle u_\ell)\right\|_{\widetilde W}\\
\leq & \sup_{s,t\in[0,1]}\|G''(\xi(s,t))\| \left(\|u_{h_{\ell-1}}-\mathcal I_{M_\ell} u_{h_{\ell-1}}\|_{W} + \|\triangle u_\ell - \mathcal I_{M_\ell} \triangle u_\ell \|_{W}\right)\|\triangle u_\ell \|_W
\end{align*}
where
\[
\xi(t,s) = \mathcal I_{M_\ell}(u_{h_{\ell-1}}+t \triangle u_\ell) + s (u_{h_{\ell-1}}-\mathcal I_{M_\ell} u_{h_{\ell-1}} + t(\triangle u_\ell - \mathcal I_{M_\ell}\triangle u_\ell)).
\]
Therefore, 
\begin{align*}
&\left\| G(u_{h_\ell})-G(u_{h_{\ell-1}})- G(\mathcal I_{M_\ell} u_{h_\ell})- G(\mathcal I_{M_\ell} u_{h_{\ell-1}}) \right\|_{\widetilde W}\\
\leq  & \phantom{+} \sup_{t\in[0,1]}\left\| G'(\mathcal I_{M_\ell}  u_{h_{\ell-1}} + t  \mathcal I_{M_\ell}\triangle u_\ell)\right\| \|\triangle u_\ell - \mathcal I_{M_\ell} \triangle u_\ell\|_W \\
& +  \sup_{s,t\in[0,1]}\|G''(\xi(s,t))\| \left(\|u_{\ell-1}-\mathcal I_{M_\ell} u_{h_{\ell-1}}\| + \|\triangle u_\ell - \mathcal I_{M_\ell} \triangle u_\ell \|_{W}\right)\|\triangle u_\ell \|_W 
\end{align*}
Taking maxima on both sides then produces the bound \eqref{eqn:bound_mlv_by_mlu} with
\[
C_{G'} =  \sup_{t\in[0,1]} \left\| G'(\mathcal I_{M_\ell}  u_{h_{\ell-1}} + t  \mathcal I_{M_\ell}\triangle u_\ell)\right\| \text{ and } C_{G''} = \sup_{s,t\in[0,1]}\|G''(\xi(s,t))\|.
\]
\end{proof}

Since $\|\triangle u_\ell\|_{C^0(\Gamma,W)}\leq \|\triangle u_\ell\|_{\mix,k}\leq \varphi(\triangle u_\ell)$ for $k\in\N\cup\{\infty\}$, we can further bound the error in \eqref{eqn:bound_mlv_by_mlu} in terms of the generic sampling error \eqref{ass:sample_error}. Indeed, 
\[
\big( C_{G'} + C_{G''}\|\triangle u_\ell\|_{C^0(\Gamma,W)} \big) \|\triangle u_\ell - \mathcal I_{M_\ell} \triangle u_\ell\|_{C^0(\Gamma,W)}\leq c \log(M_\ell)^{\mu_1}M_\ell^{-\mu_2}\varphi(\triangle u_\ell)
\]
while
\[
C_{G''}  \|\triangle u_\ell\|_{C^0(\Gamma,W)} \|u_{h_{\ell-1}}-\mathcal I_{M_\ell} u_{h_{\ell-1}}\|_{C^0(\Gamma,W)} \leq C_{G''} \tilde c_3\log(M_\ell)^{\mu_1}M_\ell^{-\mu_2}\varphi( u_{h_\ell})\varphi(\triangle u_\ell).
\]
Combining these two estimates finally allows us to write
\begin{equation}\label{eqn:ml_sampling_error_ub}
\left\|\left(\QQQ_{h_\ell} - \QQQ_{h_{\ell-1}}\right) - \left(\QQQ_{M_\ell,h_\ell}^{\mathrm{SC}} -\QQQ_{M_\ell, h_{\ell-1}}^{\mathrm{SC}}\right)\right\|_{\widetilde W}\leq c_3 \log(M_\ell)^{\mu_1}M_\ell^{-\mu_2}\varphi(\triangle u_\ell).
\end{equation}

\vspace{1em}

The optimal allocation sub-problem \eqref{eqn:optimization_general} can therefore be approximated by
\begin{equation}\label{eqn:optimization_general_a_priori}
\begin{split}
& \min_{\MMM_0,...,\MMM_L} \sum_{\ell = 0}^L M_\ell C_\ell,
\ \text{ subject to}\ \  c_3 \sum_{\ell=0}^L \log(M_\ell)^{\mu_1}M_\ell^{-\mu_2}\varphi(\triangle u_\ell) \leq \frac{\varepsilon}{2}.
\end{split}
\end{equation}

Like the single-level sampling methods, the multi-level Algorithm \ref{alg:basic_ml} is amenable to parallel implementation, the effect of which can be incorporated into the total cost by simply dividing throughout by the batch size $N_\mathrm{batch}$. Since the inclusion of this factor does not change the optimization problem \eqref{eqn:optimization_general}, we leave it out for simplicity.

\vspace{1em}

In general, problem \eqref{eqn:optimization_general} is not solved exactly, but rather formulae for $\MMM_0,...,\MMM_L$ are derived heuristically, either based on the equilibration of errors \cite{Graham2011, Harbrecht2013} or on a continuum approximation \cite{Charrier2011, Cliffe2011, Teckentrup2012}. We pursue the latter approach, i.e. to determine the optimal sample sizes, we assume for the moment that the variables $\MMM_0,\hdots,\MMM_L$ are continuous. The continuous optimization problem has relatively few variables, since $L$ is usually not too large. If in addition, the error estimates are approximated numerically, based on the general form of the generic estimate \eqref{ass:sample_error}, explicit formulae can be derived for the minimizers $\MMM_0,\hdots,\MMM_L$ in problem \eqref{alg1_optimization}, which are rounded up to the nearest admissible sample sizes. We discuss this `binning' procedure after the optimal sample sizes are derived.

\vspace{1em}

We are now in a position to estimate the optimal sample sizes $\MMM_0,\MMM_1,...,\MMM_L$ needed for our multilevel algorithm.  Again, we find it convenient to differentiate between sampling errors with- and without a logarithmic term. 

\subsection{Optimal Sample Sizes when $\mu_1=0$}
If the sampling error estimate in \eqref{eqn:ml_sampling_error_ub} is of the form 
\begin{equation}\label{eqn:sampling_error_alg}
\left\|\left(\QQQ_{h_\ell} - \QQQ_{h_{\ell-1}}\right) - \left(\QQQ_{M_\ell,h_\ell}^{\mathrm{SC}} -\QQQ_{M_\ell, h_{\ell-1}}^{\mathrm{SC}}\right)\right\|_{\widetilde W}\leq c_3 \MMM_\ell^{-\mu_2}\varphi(\triangle u_\ell)
\end{equation}
then optimization problem \ref{eqn:optimization_general_a_priori} is given by
\begin{equation}\label{eqn:optimization_alg}
\min_{\MMM_0,...,\MMM_L} \sum_{\ell=0}^L M_\ell\C_\ell , \ \ \text{subject to } c_3 \sum_{\ell=0}^L M_\ell^{-{\mu_2}}\varphi(\triangle u_\ell) \leq \frac{\varepsilon}{2}.
\end{equation}
Since the cost functional is simply a hyperplane and the constraint set is convex in $\R^{L+1}$, a unique minimizer of \eqref{eqn:optimization_alg} exists and can be readily determined via Lagrange multipliers. Moreover, at the optimum the constraint is clearly active. The Lagrangian then takes form 
\[
\mathcal L(\MMM_0,...,\MMM_L;\lambda):=\sum_{\ell=0}^L \mathcal C_\ell\MMM_\ell + \lambda \left( c_3\sum_{\ell=0}^L \MMM_\ell^{-\mu_2}\varphi(\triangle u_\ell) - \frac{\varepsilon}{2}\right),
\]
and its stationary points, obtained by letting $\frac{\partial\mathcal L}{\partial \MMM_\ell} = 0$ for $\ell=0,...,L$, satisfy
\[
\mathcal \C_\ell - \lambda c_3 \mu_2 \MMM_\ell^{-(\mu_2+1)}\varphi(\triangle u_\ell) = 0 \Rightarrow \MMM_\ell = \left(\frac{c_3\lambda \mu_2 \varphi(\triangle u_\ell)}{\C_\ell}\right)^{\frac{1}{\mu_2 + 1}}.
\]
Enforcing the equality constraint, 
\begin{align*}
\frac{\varepsilon}{2} = c_3 \sum_{\ell = 0}^L \MMM_\ell^{-\mu_2} \varphi(\triangle u_\ell) =  c_3 \sum_{\ell = 0}^L  \varphi(\triangle u_\ell) \left(\frac{c_3\lambda \mu_2 \varphi(\triangle u_\ell)}{\mathcal C_\ell}\right)^{-\frac{\mu_2}{\mu_2 + 1}}
\end{align*}
gives 
\[
(\lambda \mu_2)^{\frac{1}{\mu_2 + 1}} = \left(\frac{2}{\varepsilon}\sum_{\ell=0}^L (c_3\mathcal C_\ell^{\mu_2} \varphi(\triangle u_\ell))^{\frac{1}{\mu_2+1}}\right)^{\frac{1}{\mu_2}}
\]
and hence
\begin{equation}\label{eqn:optimization_alg_opt_sample}
\MMM_\ell = (2c_3\varepsilon^{-1})^\frac{1}{\mu_2}\left(\sum_{\ell' =0}^L (\mathcal C_{\ell'}^{\mu_2} \varphi(\triangle u_{\ell'}))^\frac{1}{\mu_2+1}\right)^{\frac{1}{\mu_2}} \left(\frac{\varphi(\triangle u_\ell)}{\mathcal C_\ell}\right)^{\frac{1}{\mu_2+1}}, \ \ \text{for }\ell = 0,...,L.
\end{equation}%
With this choice of $\MMM_0,...,\MMM_L$, the total cost satisfies%
\begin{align}
\sum_{\ell=0}^L \C_\ell \MMM_\ell &= \sum_{\ell=0}^L \C_\ell(2c_3\varepsilon^{-1})^\frac{1}{\mu_2} \left(\sum_{\ell' =0}^L (\C_{\ell'}^{\mu_2} \varphi(\triangle u_{\ell'}))^\frac{1}{\mu_2+1}\right)^{\frac{1}{\mu_2}} \left(\frac{\varphi(\triangle u_\ell)}{\C_\ell}\right)^{\frac{1}{\mu_2+1}}\nonumber\\
&= (2c_3\varepsilon^{-1})^\frac{1}{\mu_2} \left(\sum_{\ell =0}^L (\C_{\ell}^{\mu_2} \varphi(\triangle u_{\ell}))^\frac{1}{\mu_2+1}\right)^{\frac{\mu_2+1}{\mu_2}}.\label{eqn:optimization_alg_total_cost}
\end{align}

\subsection{Optimal Sample Sizes when $\mu_1>0$}

To obtain the candidate sample sizes $\MMM_0,...,\MMM_L$ in this case, we write down the optimization problem again, this time with the sampling error involving a logarithmic term%
\begin{equation}\label{eqn:optimization_log}
\min_{\MMM_0,...,\MMM_L > 1} 
\sum_{\ell=0}^L\mathcal C_\ell\MMM_\ell,
\ \text{ subject to }
\ c_3 \sum_{\ell=0}^L 
\log(\MMM_\ell)^{\mu_1}\MMM_\ell^{-\mu_2} \varphi(\triangle u_\ell) \leq \frac{\varepsilon}{2}. 
\end{equation}
Here we assume that $\frac{\varepsilon}{2}\leq \varphi(v_0)$.
We form the Lagrangian 
\[
\mathcal L(\MMM_0,...,\MMM_L;\lambda): = \sum_{\ell=0}^L \mathcal C_\ell \MMM_\ell + \lambda \left(c_3\sum_{\ell=0}^{L}\log(\MMM_\ell)^{\mu_1} \MMM_\ell^{-\mu_2}\varphi(\triangle u_\ell) - \frac{\varepsilon}{2}\right),
\]
whose stationary points satisfy
\begin{equation*}
\mathcal C_\ell + c_3 \lambda \varphi(\triangle u_\ell)\left(-\mu_2 \MMM_\ell^{-(\mu_2+1)}\log(\MMM_\ell)^{\mu_1} + \mu_1\log(\MMM_\ell)^{\mu_1-1}\MMM_\ell^{-(\mu_2+1)}\right)=0
\end{equation*}
and hence
\begin{equation}
\left( \mu_2 - \frac{\mu_1}{\log(\MMM_\ell)}\right)\MMM_\ell^{-(\mu_2+1)} \log(\MMM_\ell)^{\mu_1} = \frac{\mathcal{C}_\ell}{\lambda c_3 \varphi(\triangle u_\ell)}.
\end{equation}
In order to obtain an idea of what $\lambda$ should be, we ignore the one term consider the approximation
\begin{equation}\label{eqn:optimization_log_heuristic}
\MMM_\ell^{-(\mu_2+1)}\log(\MMM_\ell)^{\mu_1} \approx \frac{\mathcal C_\ell}{c_3\lambda \varphi(\triangle u_\ell)}.
\end{equation}
We now choose $\lambda > 0$ to ensure
\begin{equation}\label{eqn: lambda condition}
\sum_{\ell = 0}^L c_3 \varphi(\triangle u_\ell) \left(\frac{\mathcal C_\ell}{\lambda c_3\varphi(\triangle u_\ell)}\right)^{\frac{\mu_2}{\mu_2+1}} = \frac{\varepsilon}{2},
\end{equation}
i.e.
\begin{equation}\label{eqn: lambda definition}
\lambda = \left(\frac{2}{\varepsilon}\sum_{\ell=0}^L (c_3 \mathcal C_\ell^{\mu_2}\varphi(\triangle u_\ell))^{\frac{1}{\mu_2+1}} \right)^{\frac{\mu_2+1}{\mu_2}}.
\end{equation}
Note that
\[
\frac{\mathcal C_\ell}{\lambda c_3 \varphi(\triangle u_\ell)} < 1.
\]
If this were not the case, then \eqref{eqn: lambda condition} would imply 
\begin{align*}
\frac{\varepsilon}{2} &= 
\sum_{\ell = 0}^L c_3 \varphi(\triangle u_\ell) \left(\frac{\mathcal C_\ell}{\lambda c_3\varphi(\triangle u_\ell)}\right)^{\frac{\mu_2}{\mu_2+1}} \geq \sum_{\ell = 0}^L \varphi(\triangle u_\ell)
\end{align*}
(recall that we have assumed $c_3\geq 1$ w.l.o.g.) and hence $\varphi(\triangle u_\ell)<\frac{\varepsilon}{2}$ for all $\ell = 0,...,L$. In particular, $\varphi(u_0)\leq \frac{\varepsilon}{2}$, which is impossible by assumption. Inspired by Lemma \ref{lemma: epsilon cost for logs}, we now choose the sample sizes $\{\MMM_\ell\}_{\ell=0}^L$ to be%
\begin{equation} \label{eqn:optimization_log_opt_sample}
\MMM_\ell = \left\lceil\left(\frac{K_1\mathcal C_\ell}{\lambda c_3\|\triangle u_\ell\|}\right)^{-\frac{1}{\mu_2+1}}\log\left(\left(\frac{K_1\mathcal C_\ell}{\lambda c_3 \varphi(\triangle u_\ell)}\right)^{-1}\right)^{\frac{\mu_1}{\mu_2}}\right\rceil,
\end{equation}%
where $K_1$ is the scaling factor given in \eqref{eqn: smaller epsilon} and apply Lemma \ref{lemma: epsilon cost for logs} to conclude%
\begin{equation}
\MMM_\ell^{-\mu_2} \log(\MMM_\ell)^{\mu_1} \leq \left(\frac{\mathcal C_\ell}{\lambda c_3 \varphi(\triangle u_\ell)}\right)^{\frac{\mu_2}{\mu_2+1}}. 
\end{equation}%
The total multilevel sampling error can now be bounded by%
\begin{equation}
c_3\sum_{\ell=0}^L \MMM_\ell^{-\mu_2}\log(\MMM_\ell)^{\mu_1} \varphi(\triangle u_\ell)\leq c_3\sum_{\ell=0}^L \varphi(\triangle u_\ell) \left(\frac{\mathcal C_\ell}{\lambda c_3 \varphi(\triangle u_\ell)}\right)^{\frac{\mu_2}{\mu_2+1}} = \frac{\varepsilon}{2},
\end{equation}%
according to \eqref{eqn: lambda condition}. Substituting the expressions for $\{\MMM_\ell\}_{\ell=0}^L$ into the total cost then  gives%
\begin{align}
&\sum_{\ell = 0}^L \MMM_\ell \mathcal C_\ell
\leq
\sum_{\ell = 0}^L \mathcal C_\ell \left(\left(\frac{K_1\mathcal C_\ell}{\lambda c_3 \varphi(\triangle u_\ell}\right)^{-\frac{1}{\mu_2+1}}\log\left(\left(\frac{K_1\mathcal C_\ell}{\lambda c_3 \varphi(\triangle u_\ell)}\right)^{-1}\right)^{\frac{\mu_1}{\mu_2}}+1\right) \nonumber\\
= &
\left(\frac{c_3}{K_1}\right)^{\frac{1}{\mu_2+1}}\lambda^{\frac{1}{\mu_2+1}}\sum_{\ell = 0}^L \left((\mathcal C_\ell^{\mu_2}\varphi(\triangle u_\ell))^{\frac{1}{\mu_2+1}} \log\left(\left(\frac{K_1\mathcal C_\ell}{\lambda c_3 \varphi(\triangle u_\ell)}\right)^{-1}\right)^{\frac{\mu_1}{\mu_2}}\right) + \sum_{\ell=0}^L \mathcal C_\ell. \label{eqn:optimization_log_total_cost} 
\end{align}

In order to make use of formulae \eqref{eqn:optimization_alg_opt_sample} and \eqref{eqn:optimization_log_opt_sample} in Algorithm \ref{alg:basic_ml}, the sample sizes $\MMM_0,...,\MMM_L$ must first be rounded up, either to the nearest integer in the case of Monte Carlo sampling, or to the size of the sparse grid on the next refinement level $\nu$ in the case of sparse grid stochastic collocation. Since the number of additional sample points needed for the latter sampling scheme grows increasingly with increasing $\nu$, especially in high dimensions $N$, this `binning' could add needlessly to the cost. Let $\MMM^{\mathrm{next}}_0,...,\MMM_L^{\mathrm{next}}$ be the sample sizes on the next stochastic refinement level $\nu$ and $\MMM_0^{\mathrm{prev}},...,\MMM_L^{\mathrm{prev}}$ be those on the previous level $\nu-1$. The effect of `binning' can be mitigated by sorting $\{\MMM_\ell\}_{\ell=0}^{L}$ in ascending order according to the cost $(\MMM_\ell^{\mathrm{next}} - \MMM_\ell^{\mathrm{prev}})\C_\ell$ and rounding up the $\MMM_\ell$'s with lowest cost incrementally, while rounding down the others until the sampling error estimate is within tolerance.

\vspace{1em}
 
The derivations for the optimal sample sizes $\MMM_1,...,\MMM_L$ are based on the approximation of problems \eqref{eqn:optimization_alg} and \eqref{eqn:optimization_log} by their continuous counterparts, as well as other, heuristic approximations, such as \eqref{eqn:optimization_log_heuristic}. In order to to show that the multilevel algorithm leads to an improvement in efficiency over related single level methods, we need to determine its $\varepsilon$-cost. Theorem \ref{thm:multilevel_efficiency} accomplishes this. Its proof hinges on the fact that \[
\varphi(\triangle u_\ell)\leq c_4 h_\ell^\beta
\] 
for some $\beta>0$ and $c_4 \geq 1$. Therefore the sampling error for numerical integration of the correction terms $\triangle u_\ell$, decreases as the spatial refinement level $\ell$ increases. If the finite element approximation converges in mean square, this condition can easily be shown to hold for Monte Carlo sampling, but it requires a proof for Lagrange interpolation, when $\varphi(\cdot) =\|\cdot\|_{\mix,k}$. The following lemma shows that under the stricter regularity Assumption \ref{ass:coefficient_regularity_more} and under piecewise linear finite element approximation, such estimates are also possible in this case.  
\begin{assumption}\label{ass:coefficient_regularity_more}
Assume that $a(\vecy)\in C^1(D), f(\vecy)\in L^2(D)$ a.e. on $\Gamma$ and that
\begin{equation*}
\|\partial_{y_n}^k a(\vecy)\|_{L^\infty(D)}\leq \frac{\sqrt{a_{\min}}}{C_{\mathrm{reg}}} \left(\frac{\theta_n}{8}\right)^k k! \ \ \text{and}\ \ \|\partial_{y_n}^k\nabla a(\vecy)\|_{L^\infty(D)}\leq \sqrt{a_{\min}} \left(\frac{\theta_n}{8}\right)^k k!,
\end{equation*}
while 
\begin{equation*}
\|\partial_{y_n}^k f(\vecy)\|_{L^2(D)} \leq \frac{a_{\min}}{C_\mathrm{P}}(1+\|f(\vecy)\|_{L^2(D)})\left(\frac{\theta_n}{4}\right)^k k!
\end{equation*} 
where $a_{\min}\leq \sqrt{a_{\min}}< 1$ w.l.o.g., and $C_{\mathrm{reg}} \geq 1$ is a constant related to the spatial regularity of $u$ and $C_{\mathrm{P}}\geq 1$ is a Poincar\'e constant.
\end{assumption}

\begin{lemma}
Suppose the parameters $a$ and $f$ appearing in the elliptic equation \eqref{eqn:elliptic_weak_y} satisfy Assumption \eqref{ass:coefficient_regularity_more} and also that $h_\ell \leq C_{\mathrm{refine}} h_{\ell-1}$ for $\ell = 0,...,L$. Then there exists a constant $c_4\geq 1$ so that
\begin{align*}
\|\triangle u_\ell\|_{\mix,k}&\leq c_4 h_{\ell} \ \ \text{for } k\in \N\cup\{\infty\}, \ \ell = 1,2,...
\end{align*} 
\end{lemma}
\begin{proof}
It was shown in \cite{Babuska2007} (Lemma 4.4) that for every $\vecy=(y_n,\vecy_n^*)\in \Gamma$, the $k^{th}$ derivatives $\partial^k_{y_n}u$, $k\in \N_0$, are well defined as solutions of the variational problem: 
\begin{equation}\label{eqn:var_prob_der}
B(\vecy;\partial_{y_n}^k u,w) = -\sum_{l=1}^k \partial_{y_n}^l B(\vecy;\partial_{y_n}^{k-l}u,w) + (\partial_{y_n}^k f(\vecy),w),\ \ \forall w\in H^1_0(D),
\end{equation}
where
\[
B(\vecy;u,w) = \int_D a(\vecy) \nabla u \cdot \nabla w \; dx, \ \text{ and } \ (f(\vecy),w) = \int_D f(\vecy) w\; d\bx, \ \ \forall u,w\in H^1_0(D).  
\]
Moreover, they can be used to define a power series expansion $u:\mathbb C\rightarrow C^0(\Gamma_n^*;H^1_0(D))$,
\[
u(\bx,z,\vecy_n^*) = \sum_{k=0}^\infty \frac{(z-y_n)^k}{k!}\partial_{y_n}^{k}u(x,y_n,\vecy_n^*)
\] 
that converges whenever $z\in \Sigma(\Gamma_n,\tau_n) = \{z\in \mathbb C: |z-y_n|\leq \tau_n <1/(2\theta_n)\}$. The same construction holds for the Galerkin projection $u_h$ of $u$, in which case the derivatives $\partial^k_{y_n} u_h$ satisfy \eqref{eqn:var_prob_der} on $W_h(D)\subset H^1_0(D)$. It then follows readily that $\triangle u_\ell$ has the power series expansion
\[
\triangle u_\ell(\bx,z,\vecy_n^*) = \sum_{k=0}^\infty \frac{(z-y_n)^k}{k!}\partial_{y_n}^{k}\triangle u_\ell(\bx,y_n,y_n^*), \ \ \forall |z-y_n|\leq \tau_n
\]
and that to estimate $\|\triangle u\|_{\mix,\infty}$ requires bounding the terms $\|\partial_{y_n}^k\triangle u_\ell(y)\|_{H^1_0}$ for $k\in \N_0$. Let $\left(\partial_{y_n}^k u\right)_h$ denote the Galerkin projection of $\partial_{y_n}^k u$ in \eqref{eqn:var_prob_der}, i.e.
\begin{equation}\label{eqn:var_prob_der_approx}
B(\vecy;\left(\partial_{y_n}^k u\right)_h,w)=-\sum_{l=1}^k {k \choose l}\partial_{y_n}^{l}B(\vecy;\partial_{y_n}^{k-l}u, w) + (f(\vecy),w), \ \ \forall w\in W_h(D).
\end{equation}
The  approximation error $\|\partial_{y_n}^k u-\partial_{y_n}^k u_{h}\|_{H^1_0}$ for a generic spatial discretization level $h>0$ can be decomposed into   
\[
\|\partial_{y_n}^k (u-u_h)\|_{H^1_0}\leq \|\partial_{y_n}^k u - \left(\partial_{y_n}^k u\right)_h\|_{H_0^1}+\|\left(\partial_{y_n}^ku\right)_h - \partial_{y_n}^k u_h\|_{H_0^1}.
\]
Moreover, equations \eqref{eqn:var_prob_der} and \eqref{eqn:var_prob_der_approx} imply
\begin{align}
& a_{\min}\|\left(\partial_{y_n}^ku\right)_h - \partial_{y_n}^k u_h)\|_{H^1_0}^2
= - \sum_{l=1}^k {k\choose l}\partial_{y_n}^l B(\vecy;\partial_{y_n}^{k-l}(u-u_h), \left(\partial_{y_n}^ku\right)_h - \partial_{y_n}^k u_h)\nonumber\\
\leq & \sum_{l=1}^k {k\choose l}\|\partial_{y_n}^l a(\vecy)\|_{L^\infty(D)} \|\partial_{y_n}^{k-l}(u-u_h)\|_{H_0^1(D)}\|\left(\partial_{y_n}^ku\right)_h - \partial_{y_n}^k u_h)\|_{H_0^1} \label{eqn:der_of_galerkin_approx_galerkin_of_der}
\end{align}
On the other hand, it follows readily from C\'ea's Lemma and the appropriate finite element interpolation theorem (see e.g.\cite{Brenner2007}, Chapter 4) that
\begin{equation} \label{eqn:galerkin_of_der_approx_der}
\left\|\left(\partial_{y_n}^ku\right)_h - \partial_{y_n}^k u\right\|_{H^1_0(D)} \leq \frac{1}{\sqrt{a_{\min}}}\min_{w\in W_h(D)}\|\partial_{y_n}^k u - w\|_{H^1_0}\leq \frac{C_{\mathrm{mesh}}}{\sqrt{a_{\min}}} h \|\partial_{y_n}^k u\|_{H^2},
\end{equation}
where the constant $C_{\mathrm{mesh}}>0$ depends only on the triangulation $\mathcal T_h$.
Combining estimates \eqref{eqn:der_of_galerkin_approx_galerkin_of_der} and \eqref{eqn:galerkin_of_der_approx_der} then gives the recursively defined error estimate
\begin{equation}\label{eqn:recursive_der_of_galerkin}
\|\partial_{y_n}^k (u-u_h)\|_{H^1_0}\leq \frac{1}{a_{\min}}\sum_{l=1}^k {k\choose l}\|\partial_{y_n}^l a(\vecy)\|_{L^\infty} \|\partial_{y_n}^{k-l}(u-u_h)\|_{H_0^1}+\frac{C_{\mathrm{mesh}}}{\sqrt{a_{\min}}} h \|\partial_{y_n}^k u\|_{H^2}.
\end{equation}
We turn first to the norm $\|\partial_{y_n}^k u\|_{H^2(D)}$. Since $a(\vecy)\in C^1(D)$, $f(\vecy)\in L^2(D)$ and $\partial D \in C^2$, elliptic regularity theory asserts that $\|u\|_{H^2(D)} \leq C_{\mathrm{reg}} \|f(\vecy)\|_{L^2(D)}$ for an appropriate constant $C_{\mathrm{reg}}>0$ that is independent of $u$ and $f$. To bound the $H^2$-norms of the higher order  derivatives $\partial_{y_n}^k u$, $k\in\N$, we proceed inductively. Suppose $\|\partial_{y_n}^{k-l} u\|_{H^2}<\infty$ for $l=1,...,k$. Then the right hand side of \eqref{eqn:var_prob_der} can be rewritten as  
\begin{align*}
 - &\sum_{l=1}^k {k\choose l} \partial_{y_n}^l B(\vecy;\partial_{y_n}^{k-l} u,w)+ (\partial_{y_n}^k f(\vecy), w)\\
 =& \int_D \left(\sum_{l=1}^k {k\choose l} 
\left( \partial_{y_n}^l\nabla a(\vecy)\cdot \nabla \partial_{y_n}^{k-l}u + \partial_{y_n}^la(\vecy) \Delta \partial_{y_n}^{k-l} u \right) + \partial_{y_n}^k f(\vecy)\right) w \;d\bx, 
\end{align*}
through integration by parts. Moreover 
\begin{align*}
&\left\|\sum_{l=1}^k {k\choose l}  \partial_{y_n}^l\nabla a(\vecy)\cdot \nabla \partial_{y_n}^{k-l}u + \partial_{y_n}^la(\vecy) \Delta \partial_{y_n}^{k-l} u  + \partial_{y_n}^k f(\vecy) \right\|_{L^2}\\
\leq & \sum_{l=1}^k {k \choose l}\left(\|\partial_{y_n}^l\nabla a(\vecy)\|_{L^\infty}\| \partial_{y_n}^{k-l}u\|_{H^1_0} + \|\partial_{y_n}^l a(\vecy)\|_{L^\infty}\|\partial_{y_n}^{k-l}u\|_{H^2}\right) + \|\partial_{y_n}^k f(\vecy)\|_{L^2}<\infty,
\end{align*} 
and hence by regularity
\begin{align}
\|\partial_{y_n}^k u\|_{H^2}\leq & C_{\mathrm{reg}} \sum_{l=1}^k {k \choose l}\|\partial_{y_n}^l a(\vecy)\|_{L^\infty}\| \partial_{y_n}^{k-l}u\|_{H^2}\ + \label{eqn:recursive_h2_part1}\\
 & C_{\mathrm{reg}}  \sum_{l=1}^k {k \choose l}\|\partial_{y_n}^l\nabla a(\vecy)\|_{L^\infty} \|\partial_{y_n}^{k-l}u\|_{H^1_0} + \|\partial_{y_n}^k f(y)\|_{L^2}.\label{eqn:recursive_h2_part2},
\end{align}
where $\|\partial_{y_n}^k u\|_{H_0^1}$ can be shown to satisfy
\begin{equation}\label{eqn:recursive_h10}
\|\partial_{y_n}^k u\|_{H_0^1}\leq \sum_{l=1}^k {k\choose l}\frac{\|\partial_{y_n}^l a(\vecy)\|_{L^\infty}}{\sqrt{a_{\min}}}\|\partial_{y_n}^{k-l}u\|_{H_0^1}+ \frac{C_\mathrm{P}}{a_{\min}}\|\partial_{y_n}^k f(\vecy)\|_{L^2},
\end{equation}
by virtue of \eqref{eqn:var_prob_der}, where $C_\mathrm{P}>0$ is the appropriate Poincar\'e constant. Note that both \eqref{eqn:recursive_der_of_galerkin} and \eqref{eqn:recursive_h2_part2}, as well as \eqref{eqn:recursive_h10} involve inequalities that are recursively defined. The following fact provides a means by which such inequalities can be resolved and is used repeatedly in sequel. Let $c,\theta>0$ be constants and $R_0,R_1,...$ a sequence of numbers. If, for $k=1,2,...$, $R_k$ satisfies
\begin{equation}\label{eqn:recursive_formula}
R_k \leq \sum_{l=1}^k \theta^l R_{k-l} + \theta^k c  \ \ \ \text{then}\ \  R_k\leq \sum_{l=1}^k \theta^l R_{k-l} + \theta^k c \leq\frac{1}{2}(2\theta)^k(R_0+c).
\end{equation}
Since Assumption \ref{ass:coefficient_regularity_more} implies $\|\partial_{y_n}^k a(\vecy)\|_{L^\infty}\leq \sqrt{a_{\min}} (\theta_n/4)^k k!$ \\
and $\|\partial_{y_n}^k f(y)\|_{L^2}\leq (1+\|f(\vecy)\|_{L^2})\min\{1,\frac{a_{\min}}{C_P}\}(\theta_n/4)^kk!$, inequality \eqref{eqn:recursive_h10} gives rise to
\begin{align*}
\frac{\|\partial_{y_n}^k u\|_{H_0^1}}{k!}&\leq \sum_{l=1}^k \left(\frac{\theta_n}{4}\right)^{l}\frac{\|\partial_{y_n}^{k-l}u\|_{H_0^1}}{(k-l)!} + \left(\frac{\theta_n}{4}\right)^k (1+\|f(y)\|_{L^2})\\
&\leq \left(\frac{\theta_n}{2}\right)^k \frac{1}{2}(\|u\|_{H^1_0} + 1 + \|f(y)\|_{L^2})
\end{align*}
while $\|\partial_{y_n}^k \nabla a(\vecy)\|_{L^\infty}\leq \frac{1}{C_{\mathrm{reg}}}(\theta_n/4)^k k!$, together with \eqref{eqn:recursive_formula} imply that expression \eqref{eqn:recursive_h2_part2} can also be bounded above by
\begin{equation}\label{eqn:recursive_inside_bound}
k!\sum_{l=1}^k \left(\frac{\theta_n}{4}\right)^{l}\frac{\|\partial_{y_n}^{k-l}u\|_{H_0^1}}{(k-l)!} + k!\left(\frac{\theta_n}{4}\right)^k (1+\|f(\vecy)\|_{L^2})\leq k!\left(\frac{\theta_n}{2}\right)^k \frac{1}{2}(\|u\|_{H^1_0} + 1 + \|f(\vecy)\|_{L^2}).
\end{equation}
Substituting \eqref{eqn:recursive_inside_bound} into \eqref{eqn:recursive_h2_part2} and noting 
$\|\partial_{y_n}^k a(\vecy)\|_{L^\infty}\leq \frac{1}{C_{\mathrm{reg}}}(\theta_n/2)^k k! $ yields
\begin{align}
\frac{\|\partial_{y_n}^k u\|_{H^2}}{k!}&\leq \sum_{l=1}^k \left(\frac{\theta_n}{2}\right)^l \frac{\|\partial_{y_n}^{k-l}u\|_{H^2}}{(k-l)!} +  \left(\frac{\theta_n}{2}\right)^k \frac{1}{2}(\|u\|_{H_0^1}+1+\|f(\vecy)\|_{L^2})\nonumber \\
&\leq \theta_n^k \left(\left(\frac{C_{\mathrm{reg}}}{2}+\frac{C_\mathrm{P}}{4 a_{\min}}+1\right)\|f(\vecy)\|_{L^2}+ 1 \right)\label{eqn:recursive_middle}.
\end{align}
Finally, noting that $\|\partial_{y_n}^k a(\vecy)\|_{L^\infty}\leq a_{\min}\theta_n^k k!$, substituting \eqref{eqn:recursive_middle} into \eqref{eqn:recursive_der_of_galerkin} and using \eqref{eqn:recursive_formula} gives 
\begin{align*}
\frac{\|\partial_{y_n}^k (u-u_h)\|_{H^1_0}}{k!}&\leq \sum_{l=1}^k \theta_n^k \frac{\|\partial_{y_n}^{k-l}(u-u_h)\|_{H_0^1}}{(k-l)!}+\theta_n^k h \tilde c_4\leq (2\theta_n)^k \frac{1}{2}(\tilde c_4 h + \|u-u_h\|_{H_0^1})\\
&\leq h (2\theta_n)^k \frac{1}{2}\left(\tilde c_4 + \frac{C_{\mathrm{mesh}}C_{\mathrm{reg}}}{\sqrt{a_{\min}}}\right),
\end{align*}
where $\tilde c_4 =\frac{C_{\mathrm{mesh}}}{\sqrt{a_{\min}}}.  \left(\left(\frac{C_{\mathrm{reg}}}{2}+\frac{C_\mathrm{P}}{4 a_{\min}}+1\right)\|f(\vecy)\|_{L^2}+ 1 \right)$. Consequently, 
\[
\|\partial_{y_n}^k \triangle u_\ell\|_{H_0^1}\leq \|\partial_{y_n}^k (u_{h_\ell}-u)\|_{H_0^1} + \|\partial_{y_n}^k (u_{h_{\ell-1}}-u)\|_{H_0^1} \leq k! c_4 (2\theta_n)^k h_\ell,
\]
where $c_4 = \frac{1+C_{\mathrm{refine}}}{2}\left(\tilde c_4 + \frac{C_{\mathrm{mesh}}C_{\mathrm{reg}}}{\sqrt{a_{\min}}}\right)$, and hence
\[
\|\triangle u_\ell\|_{\mix,k} = \max_{n = 1,...,N}\max_{y_n\in\Gamma_n}\max_{s_n\leq k}\|\partial_{y_n}^{s_n} \triangle u_\ell\|_{H^1_0} \leq k! c_4 (2\theta_n)^k h_\ell.
\]
For $k=\infty$,
\begin{align*}
\|\triangle u_\ell\|_{\mix,\infty} &:= \max_{n=1,...,N}|u|^{(n)}_{\mix,\infty}:=\max_{n=1,...,N}\max_{z\in \Sigma(\Gamma_n;\tau_n)} \|\triangle u_\ell(z)\|_{C^0(\Gamma_n^*;H^1_0)}\\
&\leq c_4 h_\ell \max_{n=1,...,N} \max_{z\in \Sigma(\Gamma_n,\tau_n)}\sum_{k=0}^\infty (2\theta_n |z-y_n|)^k.
\end{align*}
\end{proof}

\begin{theorem}[Efficiency of Multilevel Sampling Methods]\label{thm:multilevel_efficiency}
Suppose $h_\ell := h_0 s^{-\ell}$ and let the tolerance satisfy $0<\varepsilon<\min(2\varphi(v_0),1/e)$. Suppose further that there are constants $\alpha,\gamma,\mu_2, \beta > 0$, $\mu_1 \geq 0$, and $c_1,c_2,c_3,c_4 > 0$ so that
\begin{itemize}
\item[(A1)] $\|\QQQ - \QQQ_h\|_{\widetilde W}\leq c_1 h^\alpha$,
\item[(A2)] $\C_h \leq c_2 h^{-\gamma}$,
\item[(A3)] $\left\|\left(\QQQ_{h_\ell} - \QQQ_{h_{\ell-1}}\right) - \left(\QQQ_{M_\ell,h_\ell}^{\mathrm{SC}} -\QQQ_{M_\ell, h_{\ell-1}}^{\mathrm{SC}}\right)\right\|_{\widetilde W}\leq c_3 \log(M_\ell)^{\mu_1}M_\ell^{-\mu_2}\varphi(\triangle u_\ell)$, and 
\item[(A4)] $\varphi(\triangle u_\ell) \leq c_4 h_\ell^\beta$. 
\end{itemize}
We assume throughout that $\alpha < \gamma \mu_2$ and further, without loss of generality (w.l.o.g.), that $c_i\geq 1$ for $i=1,...,4$.
Then there exists an $L \in \N$ and $\{\MMM_\ell\}_{\ell = 0}^L\subset \N^L$ so that the resulting multilevel estimate $\Q_{\{\MMM_\ell\},\{h_\ell\}}^{\mathrm{ML}}$ approximates $Q$ with a total error of
\[
\|\QQQ- \QQQ_{\{\MMM_\ell\},\{h_\ell\}}^{\mathrm{MLSC}}\|\leq \varepsilon,
\]
while the total computational cost $\C(\QQQ_{\{\MMM_\ell\},\{h_\ell\}}^{\mathrm{MLSC}})$ satisfies
\begin{equation}\label{eqn: epsilon cost of multilevel quadrature}
\mathcal C(\QQQ_{\{\MMM_\ell\},\{h_\ell\}}^{\mathrm{MLSC}}) 
\leq 
\left\{ 
\begin{array}{ll}
d_1 \varepsilon^{-\frac{1}{\mu_2}-\frac{\gamma - \beta/\mu_2}{\alpha}} \log(\varepsilon^{-1})^{\frac{\mu_1}{\mu_2}}, & \text{if } \beta < \gamma \mu_2
\\
d_2 \varepsilon^{-\frac{1}{\mu_2}} \log(\varepsilon^{-1})^{1+\frac{\mu_1}{\mu_2}}, 
& \text{if } \beta = \gamma \mu_2
\\
d_3 \varepsilon^{-\frac{1}{\mu_2}}, 
& \text{if } \beta > \gamma \mu_2
\end{array}
\right., 
\end{equation}
where the constants $d_i$ may differ according to whether $\mu_1 = 0$ or $\mu_1 > 0$.
\end{theorem}

\begin{proof}
We first choose the maximum spatial refinement level $L$ large enough to ensure that the spatial approximation error satisfies
\[
\|\QQQ-\QQQ_{h_L}]\|_{\widetilde W}\leq \frac{\varepsilon}{2}.
\]
Under Assumption (A1), it suffices to take $L$ to be the smallest integer for which 
\[c_1 h_L^{\alpha} = c_1 \left(h_0 s^{-L}\right)^{\alpha}\leq \frac{\varepsilon}{2},\]
or equivalently letting $L = \left\lceil \frac{\log(2c_1h_0^{\alpha}\varepsilon^{-1})}{\alpha \log(s)}\right\rceil$, which implies
\begin{equation}\label{eqn: bounds for L}
\frac{\log(2c_1h_0^{\alpha}\varepsilon^{-1})}{\alpha \log(s)} \leq L < \frac{\log(2c_1h_0^{\alpha}\varepsilon^{-1})}{\alpha \log(s)} + 1 = \frac{\log(2c_1(h_0 s)^{\alpha}\varepsilon^{-1})}{\alpha \log(s)}.
\end{equation}
As a direct consequence,
\begin{equation}\label{eqn: bounds for 2^L}
h_0(2c_1)^{\frac{1}{\alpha}}\varepsilon^{-\frac{1}{\alpha}}\leq s^L < s h_0(2c_1)^{\frac{1}{\alpha}} \varepsilon^{-\frac{1}{\alpha}}. 
\end{equation}
We now show that choices \eqref{eqn:optimization_alg_opt_sample} and \eqref{eqn:optimization_log_opt_sample} of sample sizes have the advertised computational cost. 
As before, we first consider the multilevel sampling scheme for which the sampling error contains no logarithmic term. Recall that the total cost \eqref{eqn:optimization_alg_total_cost} associated with formula \eqref{eqn:optimization_alg_opt_sample} satisfies
\[
\sum_{\ell=0}^L \C_\ell \MMM_\ell = (2c_3\varepsilon^{-1})^\frac{1}{\mu_2} \left(\sum_{\ell =0}^L (\C_{\ell}^{\mu_2} \varphi(\triangle u_{\ell}))^\frac{1}{\mu_2+1}\right)^{\frac{\mu_2+1}{\mu_2}}.
\]
Seeing that the sum $\sum_{\ell=0}^L (\mathcal C_\ell^{\mu_2} \varphi(\triangle u_\ell))^{\frac{1}{\mu_2+1}}$ appears frequently in sequel, it is useful to first estimate its upper bound in terms of $\varepsilon$. Under Assumptions (A2) and (A4),   
\begin{align}
\sum_{\ell = 0}^L (\mathcal C_\ell^{\mu_2} \varphi(\triangle u_\ell))^\frac{1}{\mu_2+1} 
&\leq
(c_2^{\mu_2}c_4)^{\frac{1}{\mu_2 + 1}}\sum_{\ell = 0}^L h_\ell^{\frac{\beta-\mu_2\gamma}{\mu_2+1}}\nonumber \\
&= 
(c_2^{\mu_2}c_4 h_0^{\beta-\mu_2\gamma})^{\frac{1}{\mu_2 + 1}} \sum_{\ell = 0}^{L} s^{-\frac{(\beta - \mu_2\gamma)}{\mu_2+1}\ell}\label{eqn: geom series bound initial}.
\end{align}
The upper bound for the geometric series $\sum_{\ell = 0}^{L} s^{-\frac{(\beta - \mu_2\gamma)}{\mu_2+1}\ell}$ depends on the sign of the quantity $\beta - \gamma \mu_2$ and we therefore treat each case separately.%
\begin{description}
\item[Case 1: $\beta < \gamma \mu_2$.] When the growth in the cost outweighs the decay of the correction terms, then the terms $s^{-\frac{\beta-\gamma\mu_2}{\mu_2+1}\ell}$ are increasing with $\ell$. We can now use inequality \eqref{eqn: bounds for 2^L} to bound the geometric series by%
\begin{align}
\sum_{\ell =0}^L s^{-\frac{\beta-\gamma \mu_2}{\mu_2+1}\ell} 
&= 
\frac{s^{\frac{\gamma \mu_2- \beta}{\mu_2+1}L}-1}{s^{\frac{\gamma \mu_2-\beta}{\mu_2+1}}-1} 
=
\frac{s^{\frac{\gamma\mu_2-\beta}{\mu_2}L}}{s^{\frac{\gamma\mu_2-\beta}{\mu_2+1}}} \left( \frac{1-s^{-\frac{\gamma \mu_2-\beta}{\mu_2+1}L}}{1-s^{-\frac{\gamma \mu_2-\beta}{\mu_2+1}}}\right) \nonumber \\
&\leq  
\frac{s^{\frac{\gamma\mu_2-\beta}{\mu_2+1}L}}{s^{\frac{\gamma\mu_2-\beta}{\mu_2+1}}} \left( \frac{1-s^{-\frac{\gamma \mu_2-\beta}{\mu_2+1}L}}{1-s^{-\frac{\gamma \mu_2-\beta}{\mu_2+1}L}}\right)= s^{\frac{\gamma\mu_2-\beta}{\mu_2+1}(L-1)} \nonumber \\
&\leq 
(2c_1h_0^\alpha\varepsilon^{-1})^{\frac{\gamma\mu_2-\beta}{\alpha(\mu_2+1)}}
= 
(2c_1h_0^\alpha)^{\frac{\gamma\mu_2-\beta}{\alpha(\mu_2+1)}}\varepsilon^{-\frac{\gamma\mu_2-\beta}{\alpha(\mu_2+1)}}.\label{eqn: geom series bound beta < gamma mu}
\end{align}
\item[Case 2: $\beta = \gamma \mu$.] In this case%
\begin{align}
\sum_{\ell =0}^L s^{-\frac{\beta-\gamma \mu_2}{\mu_2+1}\ell} &= (L + 1)\leq \frac{1}{\alpha\log(s)}\log(2c_1(h_0s^2)^\alpha\varepsilon^{-1}) \nonumber \\
&\leq \frac{1+\log(2c_1(h_0 s^2)^{\alpha}}{\alpha\log(s)} \log(\varepsilon^{-1}),\label{eqn: geom series bound beta = gamma mu}
\end{align}
since $\varepsilon < \frac{1}{e}$.
\item[Case 3: $\beta > \gamma \mu_2 $. ] In this case the terms $s^{-\frac{\beta - \gamma \mu_2}{\mu_2+1}\ell}$ are decreasing with $\ell$, and therefore the geometric series has upper bound%
\begin{align}
\sum_{\ell =0}^L s^{-\frac{\beta-\gamma \mu_2}{\mu_2+1}\ell} =  \frac{1-s^{-\frac{\beta-\gamma \mu_2}{\mu_2+1}L}}{1-s^{-\frac{\beta-\gamma \mu_2}{\mu_2+1}}} < \frac{1}{1-s^{-\frac{\beta-\gamma \mu_2}{\mu_2+1}}}.\label{eqn: geom series bound beta > gamma mu}
\end{align}
\end{description}
Combining inequality \eqref{eqn: geom series bound initial} with estimates \eqref{eqn: geom series bound beta < gamma mu}, \eqref{eqn: geom series bound beta = gamma mu} and \eqref{eqn: geom series bound beta > gamma mu} respectively, we obtain
\begin{equation}\label{eqn: geom series bound final}
\sum_{\ell = 0}^L (\mathcal C_\ell^{\mu_2}\varphi(\triangle v_\ell)^{\frac{1}{\mu_2+1}})
\leq 
\left\{
\begin{array}{ll}
\tilde d_1\varepsilon^{-\frac{\gamma\mu_2-\beta}{\alpha(\mu_2+1)}} ,& \text{if } \beta < \gamma \mu_2\\
\tilde d_2\log(\varepsilon^{-1}),& \text{if } \beta = \gamma \mu_2\\
\tilde d_3 ,& \text{if } \beta > \gamma \mu_2
\end{array}
\right.,
\end{equation}
where 
\begin{align*}
\tilde d_1 
&= 
(c_2^{\mu_2}c_4 h_0^{\beta-\mu_2\gamma})^{\frac{1}{\mu_2 + 1}} \left((2c_1h_0^\alpha)^{\frac{\gamma\mu_2-\beta}{\alpha(\mu_2+1)}}\right)
\\
\tilde d_2 
& = (c_2^{\mu_2}c_4 h_0^{\beta-\mu_2\gamma})^{\frac{1}{\mu_2 + 1}} \left(\frac{1+\log(2c_1(h_0 s^2)^{\alpha}}{\alpha\log(s)} \right)
\\
\tilde d_3
& = (c_2^{\mu_2}c_4 h_0^{\beta-\mu_2\gamma})^{\frac{1}{\mu_2 + 1}}\left(\frac{1}{1-s^{-\frac{\beta-\gamma \mu_2}{\mu_2+1}}}\right).
\end{align*}
Substituting \eqref{eqn: geom series bound final} into the total cost \eqref{eqn:optimization_alg_total_cost} now yields
\begin{align}
\sum_{\ell=0}^L \C_\ell \MMM_\ell &\leq 
\left\{\begin{array}{ll} 
(2c_3\tilde d_1^{\mu_2+1})^{\frac{1}{\mu_2}}\varepsilon^{-\frac{1}{\mu_2}-\frac{\gamma-\frac{\beta}{\mu_2}}{\alpha}} , & \text{if } \beta < \gamma \mu_2\\
(2c_3\tilde d_2^{\mu_2+1})^{\frac{1}{\mu_2}}\varepsilon^{-\frac{1}{\mu_2}}\log(\varepsilon^{-1})^{\frac{\mu_2+1}{\mu_2}}, & \text{if } \beta = \gamma \mu_2\\ 
(2c_3\tilde d_3^{\mu_2+1})^{\frac{1}{\mu_2}}\varepsilon^{-\frac{1}{\mu_2}}, & \text{if } \beta > \gamma \mu_2
\end{array}
\right. 
\end{align}
Next, we consider the total cost when the sample sizes are chosen according to \eqref{eqn:optimization_log_opt_sample}, i.e. 
\begin{align*}
\sum_{\ell = 0}^L \MMM_\ell \mathcal C_\ell
\leq &
\left(\frac{c_3}{K_1}\right)^{\frac{1}{\mu_2+1}}\lambda^{\frac{1}{\mu_2+1}}\sum_{\ell = 0}^L \left((\mathcal C_\ell^{\mu_2}\varphi(\triangle v_\ell))^{\frac{1}{\mu_2+1}} \log\left(\left(\frac{K_1\mathcal C_\ell}{\lambda c_3 \varphi(\triangle v_\ell)}\right)^{-1}\right)^{\frac{\mu_1}{\mu_2}}+\C_\ell\right).
\end{align*}

The sum $\sum_{\ell=0}^L \mathcal C_\ell$ can readily be shown to have an upper bound similar to \eqref{eqn: geom series bound beta < gamma mu}. In fact, under Assumption (A2) 
\begin{align}
\sum_{\ell=0}^L \mathcal C_\ell \leq c_2 h_0^{-\gamma}\sum_{\ell=0}^L s^{\gamma\ell}\leq (2c_1h_0^{-\alpha})^{\frac{\gamma}{\alpha}} \varepsilon^{-\frac{\gamma}{\alpha}}<(2c_1h_0^{-\alpha})^{\frac{\gamma}{\alpha}}\varepsilon^{-\frac{1}{\mu_2}},\label{eqn: cost per sample sum bound}
\end{align} 
since $\alpha < \mu_2 \gamma$.
Consider the log term
\begin{align}
&\log\left(\left(\frac{K_1\mathcal C_\ell}{\lambda c_3 \varphi(\triangle u_\ell)}\right)^{-1}\right)^{\frac{\mu_1}{\mu_2}}\nonumber \\ = &
\log\left(\left(\frac{2}{\varepsilon}\sum_{\ell'=0}^L (c_3 \mathcal C_{\ell'}^{\mu_2}\|\triangle u_{\ell'}\|)^{\frac{1}{\mu_2+1}} \right)^{\frac{\mu_2+1}{\mu_2}}\left(\frac{c_3 \varphi(\triangle u_\ell)}{K_1\mathcal C_\ell}\right)\right)^{\frac{\mu_1}{\mu_2}}\nonumber \\
= &
\log\left(K_1^{-1}(2^{\mu_2+1}c_3^{\mu_1+1})^{\frac{1}{\mu_2}}\varepsilon^{-\frac{\mu_2+1}{\mu_2}}\left(\sum_{\ell'=0}^L (\mathcal C_{\ell'}^{\mu_2}\|\triangle u_{\ell'}\|)^{\frac{1}{\mu_2+1}} \right)^{\frac{\mu_2+1}{\mu_2}}\left(\frac{\varphi(\triangle u_\ell)}{\mathcal C_\ell}\right)\right)^{\frac{\mu_1}{\mu_2}}.\label{eqn: total cost bound log term 1}
\end{align}
Since the computational cost at the lowest spatial refinement level satisfies $\mathcal C_0 \leq C_\ell$ for $\ell>0$ it follows by virtue of Assumption (A2) that
\begin{align}
\frac{\varphi(\triangle u_\ell)}{\mathcal C_\ell}
&\leq
\frac{c_2 h_\ell^{\beta}}{C_0} 
=
\frac{c_2 h_0^{\beta}s^{-\beta
\ell}}{C_0} 
\leq
\frac{c_2 h_0^{\beta}}{C_0}.\label{eqn: total cost bound log term 2}
\end{align}
Moreover, according to \eqref{eqn: geom series bound final},
\begin{equation}\label{eqn: geom series one upper bound}
\sum_{\ell'=0}^L (\mathcal C_{\ell'}^{\mu_2}\|\triangle u_{\ell'}\|)^{\frac{1}{\mu_2+1}}\leq \max_{i=1,2,3}\{\tilde d_i\}\varepsilon^{-\max\{1,\frac{\gamma\mu_2-\beta}{\alpha(\mu_2+1)}\}}.
\end{equation}
Combining \eqref{eqn: geom series one upper bound} with \eqref{eqn: total cost bound log term 2} in \eqref{eqn: total cost bound log term 1} now yields 
\begin{align}
&\log\left(\left( \frac{\mathcal C_\ell}{\lambda c_3 \varphi(\triangle u_\ell)}\right)^{-1} \right)^{\frac{\mu_1}{\mu_2}} 
\leq
\log(K_2 \varepsilon^{-K_3})^{\frac{\mu_1}{\mu_2}}
\leq 
\left(\log(K_2)+K_3\right)^{\frac{\mu_1}{\mu_2}}\log(\varepsilon^{-1})^{\frac{\mu_1}{\mu_2}},\label{eqn: total cost bound log term  final}
\end{align}
where
\begin{align*}
K_2 &= K_1^{-1}\mathcal C_0^{-1}s^\beta 2^{1+\frac{1}{\mu_2}} c_2 c_3^{\frac{\mu_1+1}{\mu_2}}(\max_{i=1,2,3}\{\tilde d_i\})^{\frac{\mu_2+1}{\mu_2}}, \ \text{and}\\
K_3 & = \left(1+\max\{1,\frac{\gamma\mu_2-\beta}{\alpha(\mu_2+1)}\}\right)\frac{\mu_2+1}{\mu_2}.
\end{align*}
Incorporating the upper bounds \eqref{eqn: geom series bound final}, \eqref{eqn: cost per sample sum bound} and \eqref{eqn: total cost bound log term final} into the total cost \eqref{eqn:optimization_log_total_cost} and using expression \eqref{eqn: lambda definition} for $\lambda$, we finally get 
\begin{align*}
&\sum_{\ell = 0}^L \mathcal C_\ell \MMM_\ell \leq 
\left(\frac{c_3}{K_1}\right)^{\frac{1}{\mu_2+1}} (\log(K_2)+K_3)^{\frac{\mu_1}{\mu_2}} \log(\varepsilon^{-1})^{\frac{\mu_1}{\mu_2}}\lambda^{\frac{1}{\mu_2+1}}\sum_{\ell=0}^L (\mathcal C_\ell^{\mu_2}\varphi(\triangle u_\ell))^{\frac{1}{\mu_2+1}}+\sum_{\ell=0}^L\mathcal C_\ell\\
\leq & 
2^{\frac{1}{\mu_2}}K_1^{-\frac{1}{\mu_2+1}} (\log(K_2)+K_3)^{\frac{\mu_1}{\mu_2}} 
c_3^{ \frac{1}{\mu_2}} \left(\sum_{\ell=0}^L (\mathcal C_\ell^{\mu_2}\varphi(\triangle u_\ell))^{\frac{1}{\mu_2+1}}\right)^{\frac{\mu_2+1}{\mu_2}}\varepsilon^{-\frac{1}{\mu_2}}\log(\varepsilon^{-1})^{\frac{\mu_1}{\mu_2}}+\sum_{\ell=0}^L\mathcal C_\ell\\
\leq & 
\left\{ 
\begin{array}{ll}
d_1 \varepsilon^{-\frac{1}{\mu_2}-\frac{\gamma - \beta/\mu_2}{\alpha}} \log(\varepsilon^{-1})^{\frac{\mu_1}{\mu_2}}, & \text{if } \beta < \gamma \mu_2
\\
d_2 \varepsilon^{-\frac{1}{\mu_2}} \log(\varepsilon^{-1})^{1+\frac{\mu_1}{\mu_2}}, 
& \text{if } \beta = \gamma \mu_2
\\
d_3\varepsilon^{-\frac{1}{\mu_2}}, 
& \text{if } \beta > \gamma \mu_2
\end{array}
\right.,
\end{align*}
where
$ \displaystyle d_i  = 2^{\frac{1}{\mu_2}}K_1^{-\frac{1}{\mu_2+1}} (\log(K_2)+K_3)^{\frac{\mu_1}{\mu_2}} 
c_3^{ \frac{1}{\mu_2}}\tilde d_i^{\frac{\mu_2+1}{\mu_2}}+(2c_1h_0^\alpha)^{\frac{\gamma}{\alpha}}$ for $i=1,2,3$.

\end{proof}

\section{Numerical Examples}\label{section:numerical_examples}

This section discusses the numerical implementation of the multilevel sparse grid algorithm described in the previous sections. We apply both the multilevel Monte Carlo and sparse grid algorithms to estimate the spatially varying mean of the solution to the elliptic equation \eqref{eqn:elliptic_parametric_y} with a random diffusion coefficient on either the unit interval, i.e. $D = [0,1]$ or the unit square, i.e. $D = [0,1]^2$. For both these spatial domains, we choose the diffusion coefficient $q$ to be the univariate random field defined at $x_1\in [0,1]$ by
\[
 \log( a(x_1,\omega) - 0.5) = 1+ \left(\frac{\sqrt{\pi}L}{2}\right)^{\frac{1}{2}} Y_1(\omega) +\sum_{n=2}^\infty b_n(x_1) Y_n(\omega),
\]
where
\[
 b_n(x_1) := \left(\sqrt{\pi}L\right)^{\frac{1}{2}}\exp\left(\frac{-(\lfloor \frac{\pi}{2}\rfloor \pi L)^2}{8}\right)\left\{ 
\begin{array}{ll} 
\sin\left(\frac{\lfloor\frac{\pi}{2}\rfloor \pi x_1}{L}\right) & \text{ if  $n$ is even,} \\
\cos\left(\frac{\lfloor\frac{\pi}{2}\rfloor \pi x_1}{L}\right) & \text{ if  $n$ is odd},
\end{array}
\right. 
\]
and the random variables $\{Y_n\}_{n=1}^\infty$ are independent and uniformly distributed over the interval $[-\sqrt{3},\sqrt{3}]$. The parameter $L$ relates to the correlation length of the field $\log(q(x,\omega) - 0.5)$. Indeed it can be shown that the covariance function
\[
 \mathrm{Cov}[\log( a - 0.5)](x_1,x_1') = \exp\left(\frac{-(x_1-x_1')^2}{L^2}\right).
\]
For short correlation lengths, finite noise approximations of $q$ require a large number of terms to accurately represent its correlation structure, leading not only to a high stochastic dimension, but also to the presence of fine scale oscillations that can only be resolved with sufficiently fine meshes (see \cite{Charrier2011}). Here we do not consider the effect of this truncation error, and take $L=0.25$ and $N=5$. We also let the deterministic forcing term $f$ to be given by $f(x_1) = \cos(x_1)$ when $D = [0,1]$, and $f(x_1,x_2)=\cos(x_1)\sin(x_2)$, when $D = [0,1]^2$. The parameters $f$ and $q$ readily satisfy the smoothness conditions made in Assumptions \ref{ass:coefficient_regularity} and \ref{ass:coefficient_regularity_more}, justifying the use of sparse grids and were in fact used in \cite{Nobile2008} to show the competitive convergence rate of sparse grid methods vis-\`a-vis Monte Carlo sampling and stochastic finite elements.

\vspace{1em}

We solve each realization of the system using the finite element method with continuous piecewise polynomial basis functions and computational cost per solve was measured in CPU time. We obtained estimates for the spatial error through the spatial $L^2$ norms of the correction terms and for the sparse grid quadrature error by comparing successive sparse grid approximations $I_\MMM[v]$ in the spatial $L^2$ norm. Since the  convergence rates of sparse grid stochastic collocation methods depend on quantities that can not readily be computed \emph{a priori}, such as the radii $\tau_n$ of the regions of analyticity, they must be estimated during the execution of the program, unlike that of the Monte Carlo method ($\mu_2 = \frac{1}{2}$). We achieve and update this estimate by generating an initial sample on the coarsest level as well as after each spatial refinement step, before computing the optimal sample sizes. An overly conservative initial sample size will generate more sample paths than are necessary, especially when the sampling scheme has a fast convergence rate, while a sample size that is too small may lead to inaccurate diagnostic parameters, both of which have a detrimental effect on the efficiency of the algorithm. To mitigate this risk, we begin with a relatively large initial sample size on the coarsest level and reduce it gradually as our confidence in the estimated convergence rate improves.

\vspace{1em}

\begin{example}[1D]
Let $D = [0,1]$ with an initial mesh of uniform subintervals of length $h = 1/8$. We use a tolerance level $\varepsilon = 10^{-3}$ and refine the mesh by scaling $h$ at each step by the factor $s = 4$. Figure \ref{fig:example1_epsilon_cost} plots the $\varepsilon$-cost for single-  and multi-level versions of both Monte Carlo sampling and sparse grid stochastic collocation, based on different spatial refinement levels. As expected, the sparse grid stochastic collocation method is more efficient than Monte Carlo sampling and in both cases the multilevel algorithm achieves a considerable speed-up. For this example, four spatial mesh refinements are required to obtain a spatial error within tolerance (see Figure \ref{fig:example1_spatial_error_linear}).

\begin{figure}[H]
\centering
\includegraphics[width=0.6\textwidth]{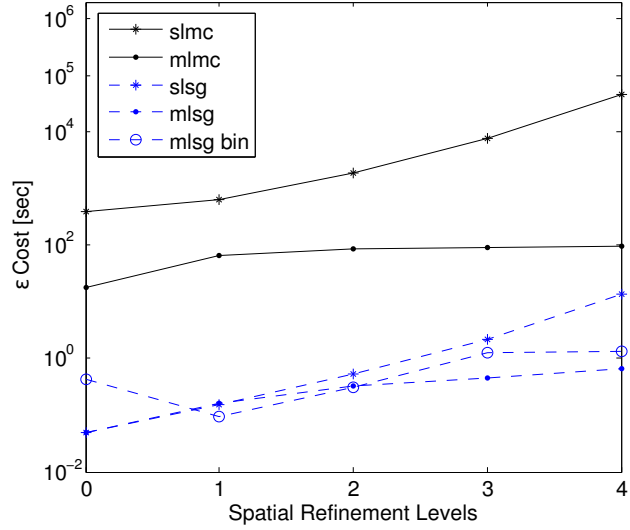}
\caption{The total $\varepsilon$-cost of the single- and multilevel Monte Carlo (slmc,mlmc) and sparse grid (slsg, mlsg, mlsg bin) methods. The dataset `mlsg' represents the computed optimal sample sizes, while `mlsg bin' refers to the binned sample sizes used to generate the actual multilevel estimate.}
\label{fig:example1_epsilon_cost}
\end{figure}

\begin{figure}[H]
\centering
\begin{subfigure}[b]{0.47\textwidth}
\includegraphics[width=\textwidth]{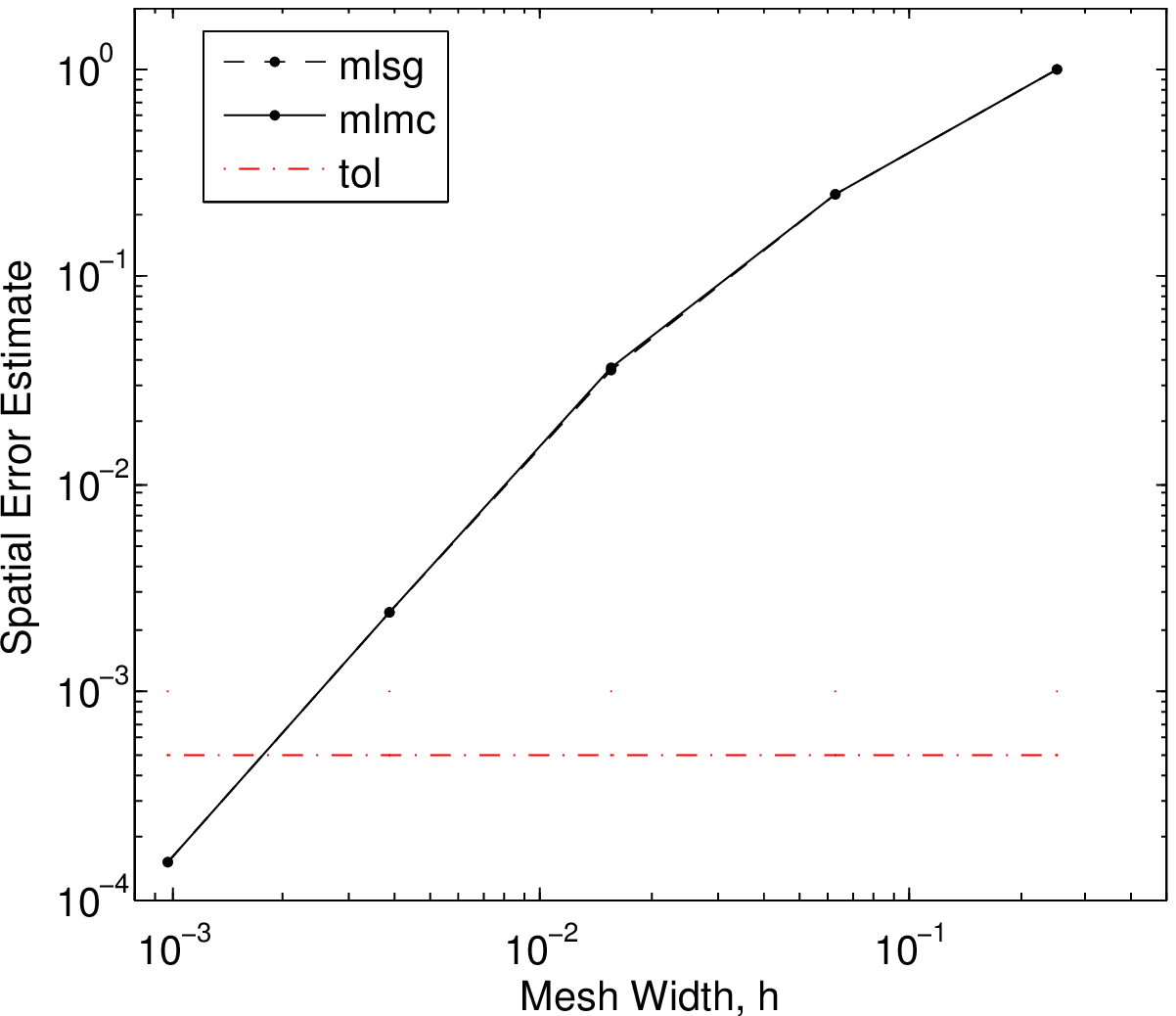}
\caption{Spatial error estimate ($\mathrm{tol} = \frac{\varepsilon}{2}$).}
\label{fig:example1_spatial_error_linear}
\end{subfigure}%
\quad%
\begin{subfigure}[b]{0.47\textwidth}
\includegraphics[width=\textwidth]{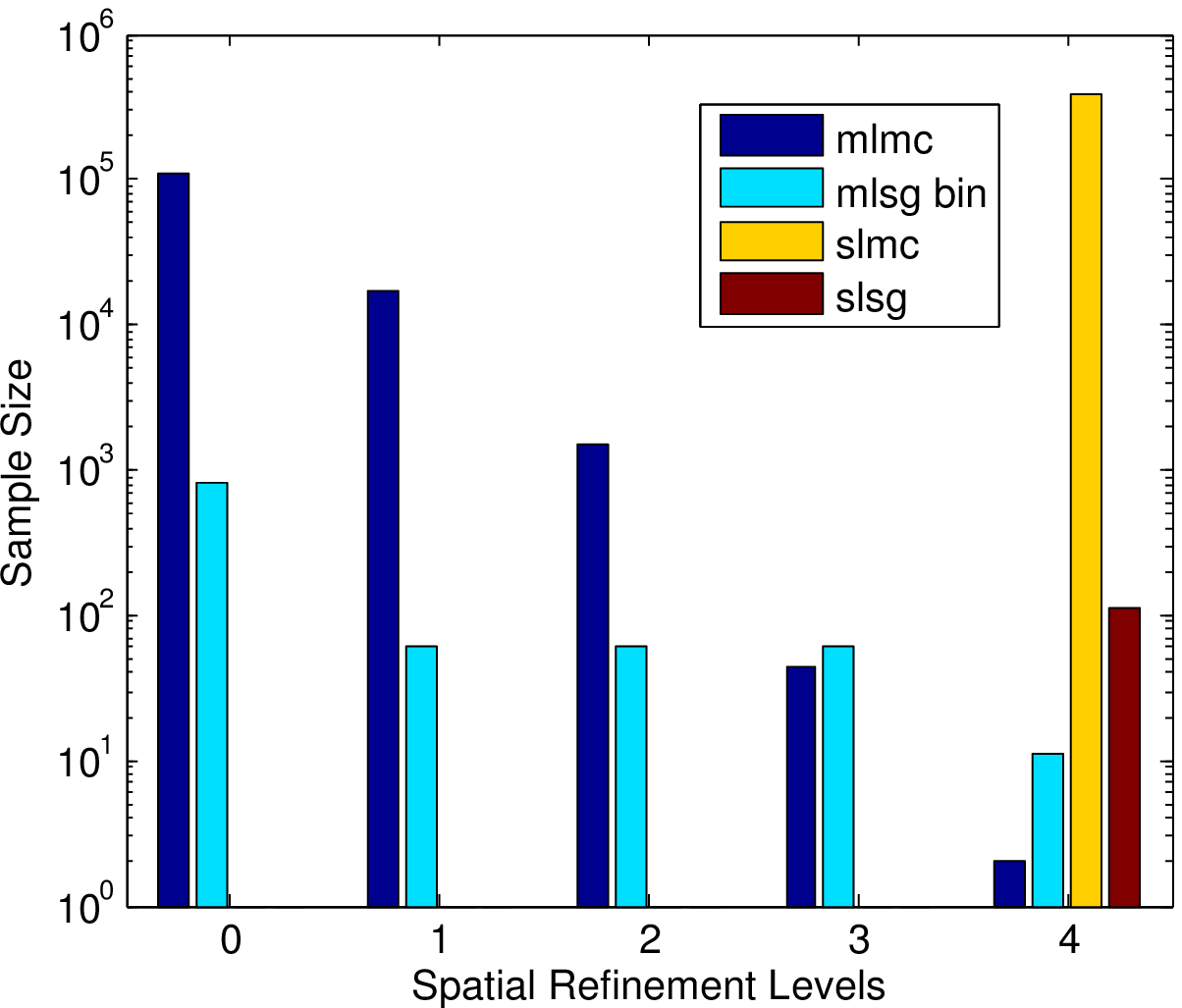}
\caption{Optimal sample sizes per level.}
\label{fig:example1_sample_sizes}
\end{subfigure}%
\caption{}\label{fig:example1_efficiency}
\end{figure}

From our analysis (Theorem \ref{thm:multilevel_efficiency}) it would seem that a faster spatial convergence rate, i.e. a higher value of $\alpha$ would improve the overall efficiency. Figure shows this to be the case for our example. Indeed not only are fewer refinement steps necessary for higher order polynomial approximation, but the computational effort also decreases.

\begin{figure}[H]
\centering
\begin{subfigure}[b]{0.47\textwidth}
\includegraphics[width=\textwidth]{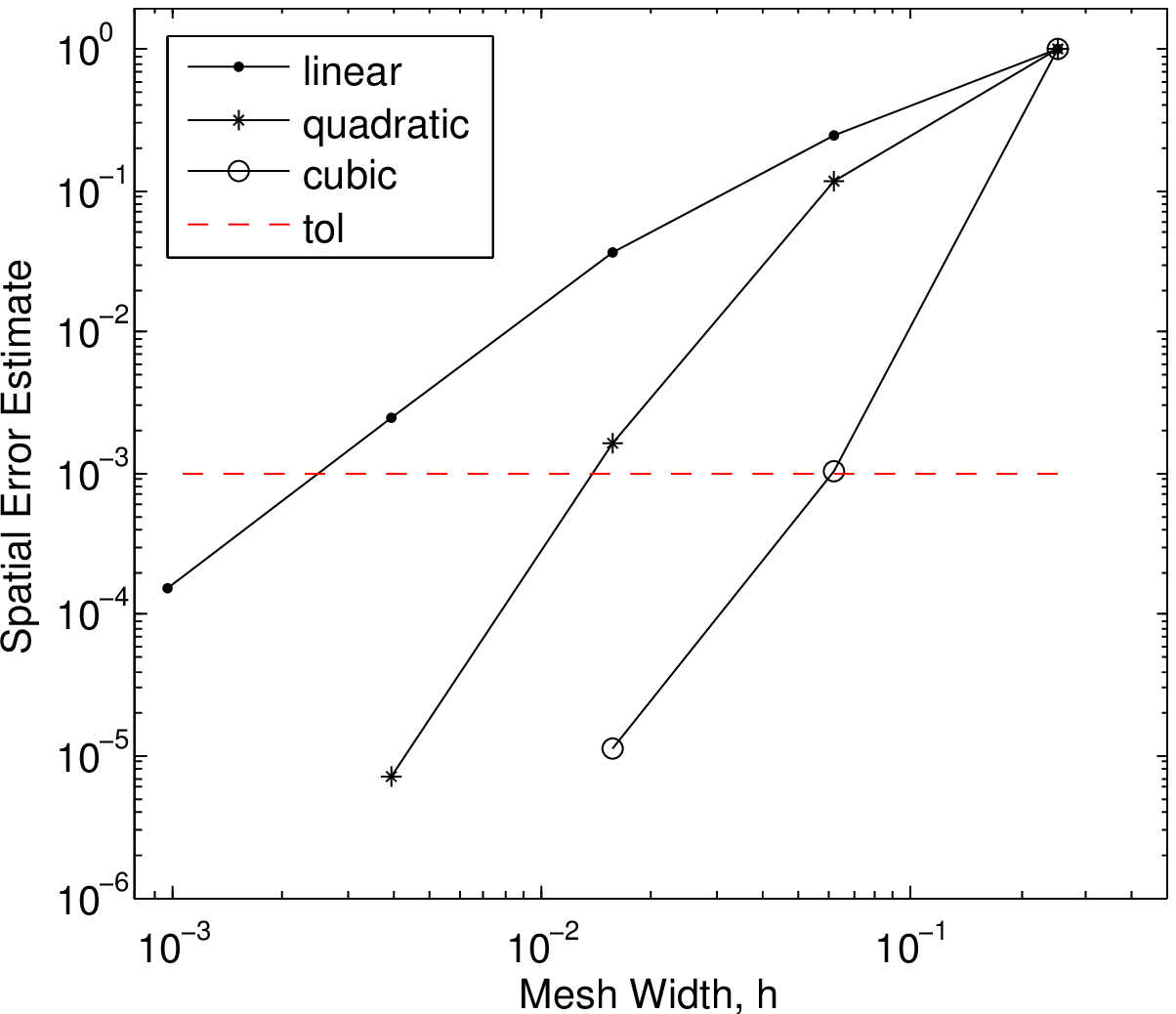}
\caption{Spatial error estimate for different order polynomial basis functions ($\mathrm{tol} = \frac{\varepsilon}{2}$).}
\label{fig:example1_spatial_error_prefine}
\end{subfigure}%
\quad%
\begin{subfigure}[b]{0.47\textwidth}
\includegraphics[width=\textwidth]{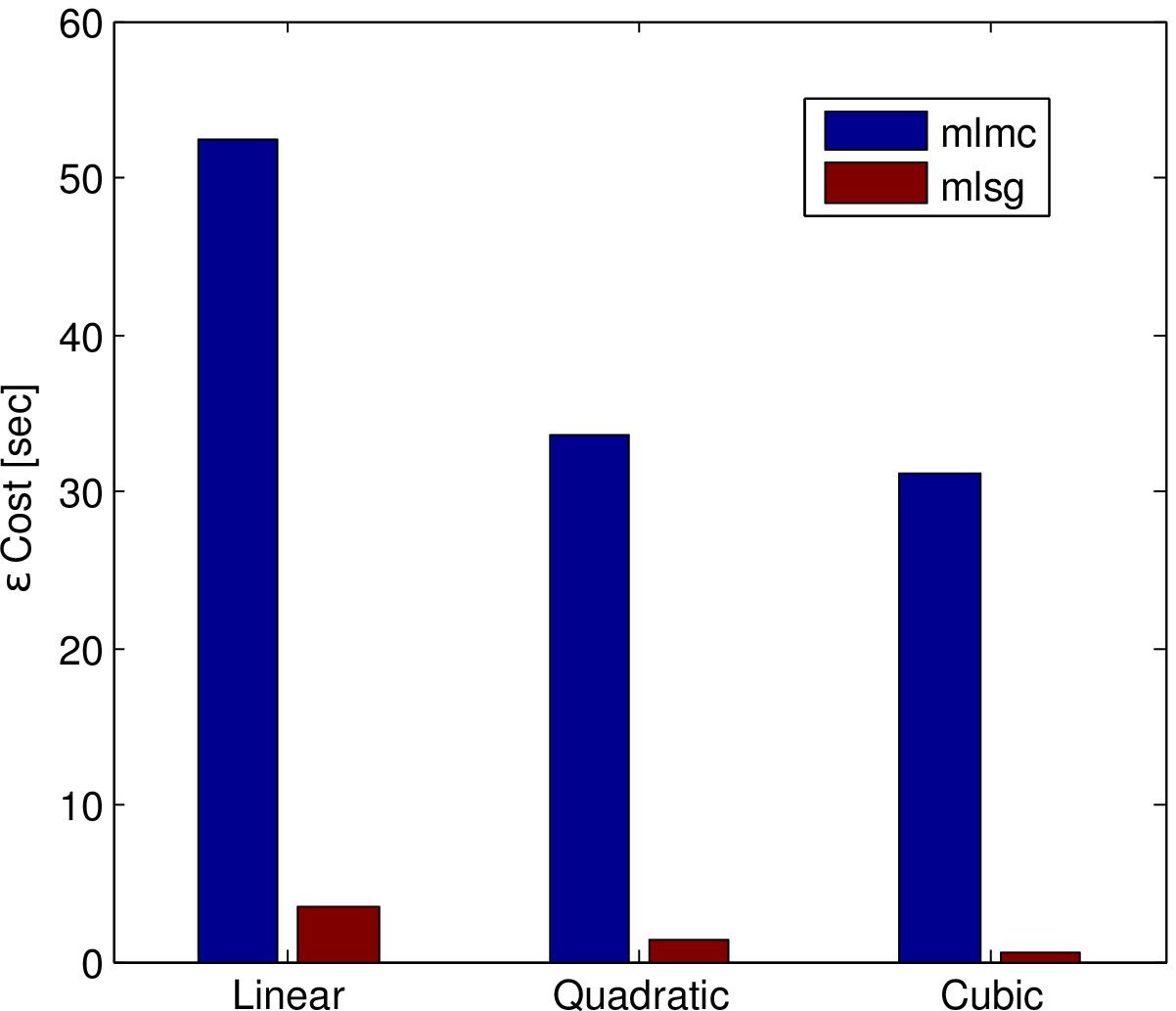}
\caption{The overall $\varepsilon$-cost of the multilevel algorithm, both for Monte Carlo sampling and sparse grid stochastic collocation.}
\label{fig:example1_epsilon_cost_smoothbasis}
\end{subfigure}%
\caption{The effect of using a higher order finite element method on the efficiency of the multilevel algorithm.}\label{fig:example1_smoother_basis}
\end{figure}
 
In order to investigate the effect of the refinement parameter $s$ and the number of spatial refinement steps needed on the algorithm's efficiency, we repeated Example 1 using linear basis functions, but with  different values of $s$, ranging from $s = 2,4,6,8,10$ to $s=160$. We computed the extreme value $s=160$, based on diagnostic information from previous examples by determining the mesh width $h$ for which the spatial error is within tolerance, so that with $s=160$ only one refinement step is necessary. We also used the previous, more accurate convergence rates to determine the optimal sample sizes. In other words, the case $s=160$ is unrealistic but was used to shed some light on the effect that the number of refinement steps has on the overall efficiency.

\begin{figure}[H]
\centering
\begin{subfigure}[b]{0.47\textwidth}
\includegraphics[width=\textwidth]{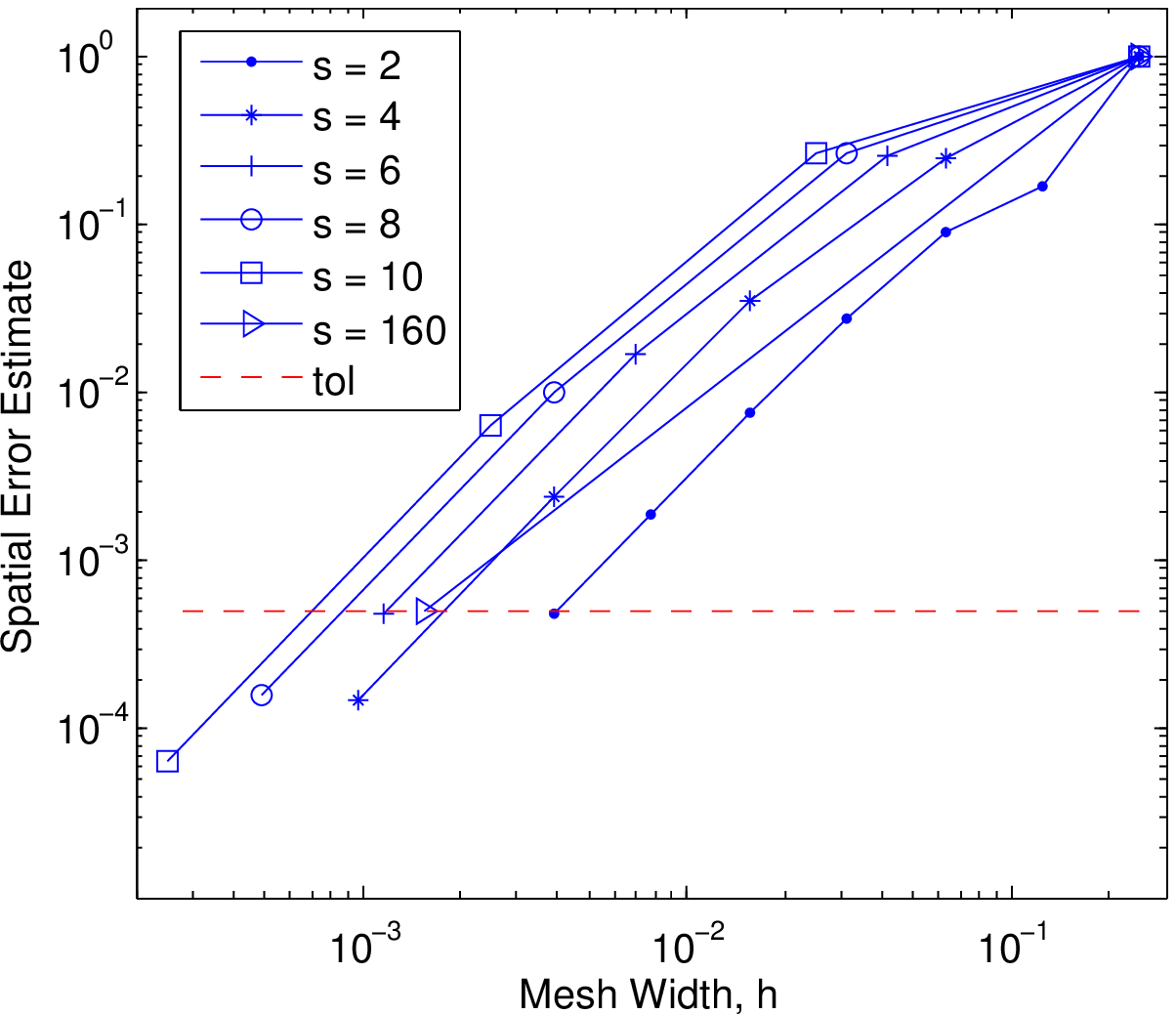}
\caption{Spatial error estimate for different values of the mesh refinement parameter $s$ ($\mathrm{tol} = \frac{\varepsilon}{2}$).}
\label{fig:example1_spatial_error_s_refine}
\end{subfigure}%
\quad%
\begin{subfigure}[b]{0.47\textwidth}
\includegraphics[width=\textwidth]{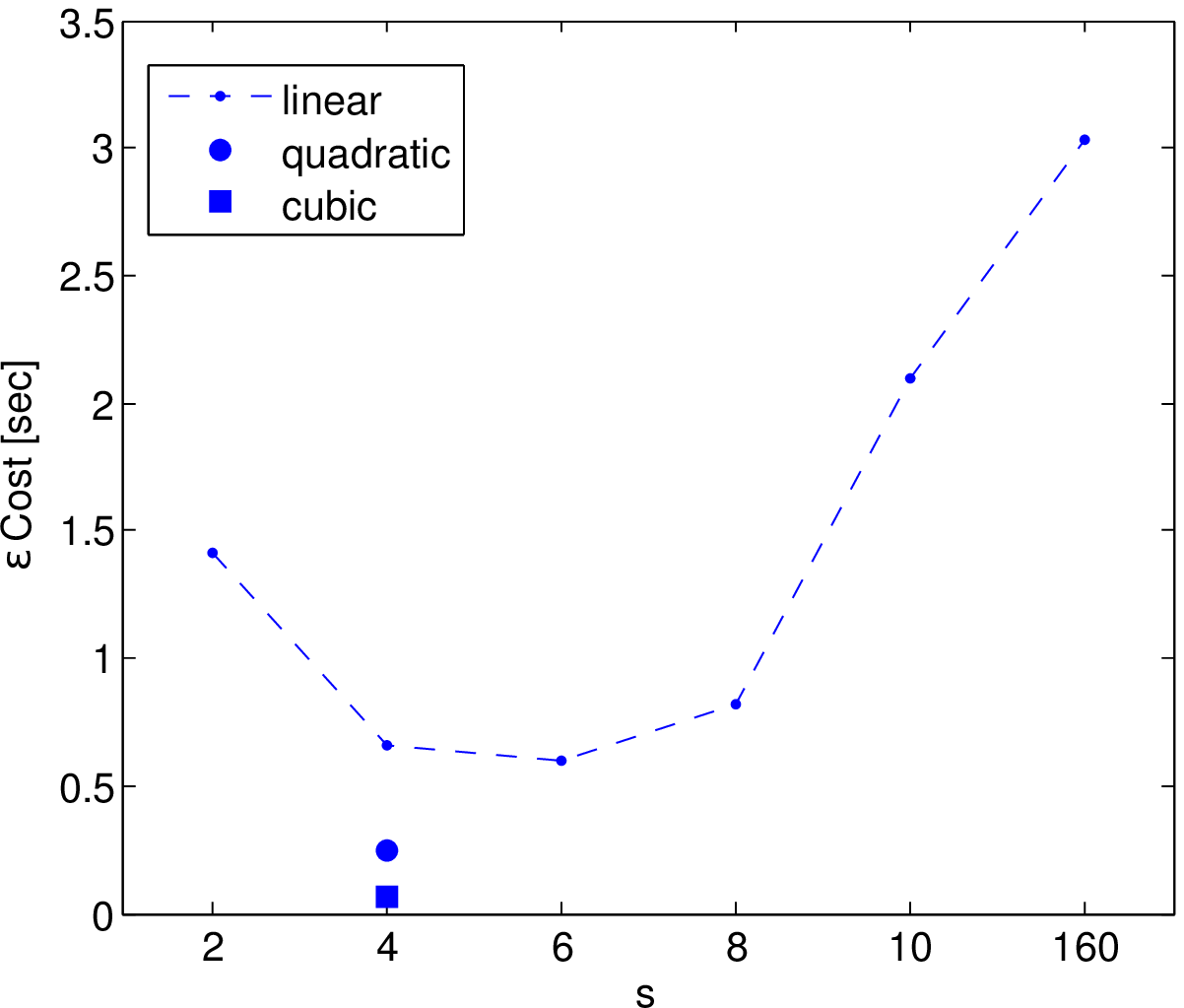}
\caption{The overall $\varepsilon$-cost of the multilevel algorithm for different values of the mesh refinement parameter $s$.}
\label{fig:example1_epsilon_cost_s_refine}
\end{subfigure}%
\caption{The effect of spatial mesh refinement on the efficiency of the multilevel algorithm.}\label{fig:example1_s_refine}
\end{figure}

The results, as summarized in Figure \ref{fig:example1_s_refine}, are not conclusive. It seems (see Figure \ref{fig:example1_epsilon_cost_s_refine}) that there is an optimal value for $s$, in this case $s=6$, for which the computational effort is minimal. More moderate refinement strategies may lead to a needlessly many levels and hence too many unnecessary samples, while those that are overly aggressive might overshoot the  mesh size $h$ required by the tolerance level (see Figure \ref{fig:example1_spatial_error_s_refine}), thereby incurring a needlessly high cost. These, however cannot be the only determinants of efficiency, since the value $s=160$, giving precisely the right $h$, would then be expected to outperform the others. In other words, the number of spatial refinement models also seems to have an influence on the overall efficiency of the algorithm. More work is needed to untangle the effect of the mesh refinement strategy on the $\varepsilon$-cost of the algorithm.

\end{example}

\begin{example}[2D]
Consider the spatial domain $D = [0,1]^2$ subdivided by uniform triangulation with mesh width $h=0.25$. Here we use the same tolerance level as before, i.e. $\varepsilon = 10^{-3}$ and refine the mesh at each step by dyadic subdivision, i.e. $s = 2$. The results are comparable to those in Example 1. The sparse grid method outperforms the Monte Carlo sampling scheme in both the single- and multilevel cases, although the multilevel Monte Carlo method is more efficient than the single level sparse grid method in this case. The degrees of freedom of the sample deterministic systems ranged from $64$ to $16641$ and in fact the maximal number of refinement steps were reached before the spatial error estimate was within tolerance. At such high refinement levels, it is not only the deterministic system solve, but also the assembly and interpolation operations that contribute significantly to the overhead. On the other hand, there is a wealth of information available from samples already generated, which could potentially be incorporated into the assembly and solution of a given system realization, thus providing a much needed speed-up.

\begin{figure}[H]
\centering
\begin{subfigure}[b]{0.47\textwidth}
\includegraphics[width=\textwidth]{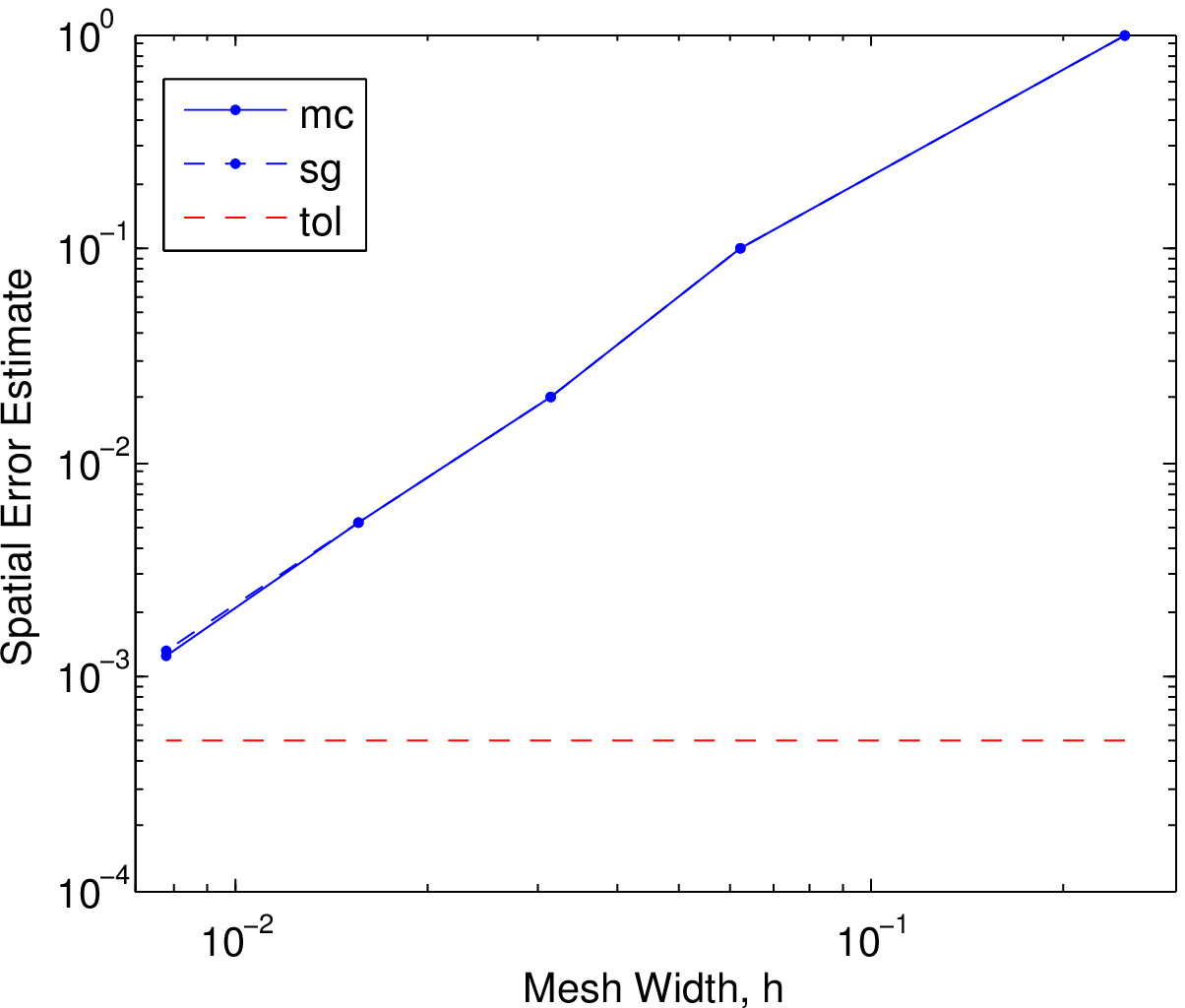}
\caption{Spatial error estimate for Monte Carlo sampling and sparse grid stochastic collocation ($\mathrm{tol} = \frac{\varepsilon}{2}$).}
\label{fig:example2_spatial_error}
\end{subfigure}%
\quad%
\begin{subfigure}[b]{0.47\textwidth}
\includegraphics[width=\textwidth]{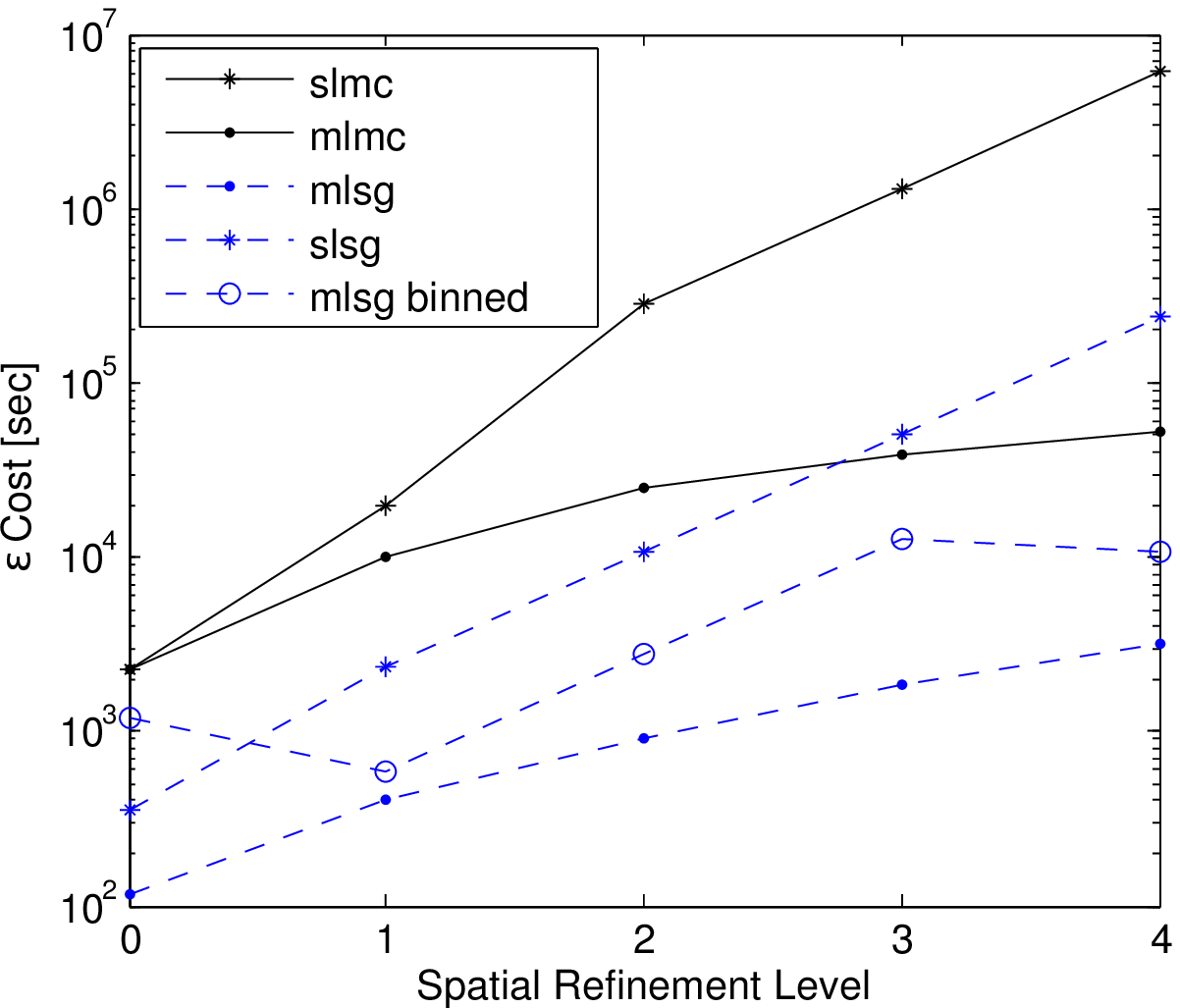}
\caption{he total $\varepsilon$-cost of the single- and multilevel Monte Carlo (slmc,mlmc) and sparse grid (slsg, mlsg, mlsg bin) methods.}
\label{fig:example2_epsilon_cost}
\end{subfigure}%
\caption{The multilevel Monte Carlo- and sparse grid algorithms for a 2D spatial problem.}\label{fig:example2_overall}
\end{figure}
 
\end{example}

\section{Discussion}\label{section:conclusion}

Multilevel sampling methods offer an improvement on the efficiency of single level methods without loosing any of their salient features, such as parallel implementation, nestedness, or non-intrusiveness. In this paper we have shown that the multilevel Monte Carlo algorithm developed in \cite{Cliffe2011} can readily be extended to interpolation-based sampling schemes (such as sparse grid stochastic collocation) leading to an even greater efficiency in certain cases. Despite the technical difficulties in proving that the multilevel algorithm improves the computational complexity, this method is surprisingly straightforward to implement if the errors and convergence rates are estimated numerically. This supports the claim that the multilevel algorithm can be used as a wrapper, coordinating the spatial refinement with the quadrature level. An area of future work would be to investigate this claim in the case of adaptive sampling schemes. Furthermore, it is not yet entirely clear how the spatial refinement strategy effects the overall performance of the algorithm, although it was seen in to have a considerable influence. Apart from improving efficiency, multilevel methods strategically record useful information that can be harnessed to further improve computation.

\bibliographystyle{elsarticle-num} 
\bibliography{multilevel_bib}

\end{document}